\documentclass[12pt,reqno]{amsart}
\usepackage{amsfonts}
\usepackage{amsmath,amssymb,url}

\def\marg#1#2{\def\marginnotetextwidth{\the\textwidth}\marginnote{\bf #1}{\bf #2}}

\numberwithin{equation}{subsection}

\newtheorem{theorem}{Theorem}[section]
\newtheorem{theorem*}{Theorem}
\newtheorem{prop}[theorem]{Proposition}
\newtheorem{lemma}[theorem]{Lemma}
\newtheorem{cor}[theorem]{Corollary}

\theoremstyle{definition}

\newtheorem{example}[theorem]{Example}
\newtheorem{definition}[theorem]{Definition}
\newtheorem*{definition*}{Definition}
\newtheorem{remark}[theorem]{Remark}

\newcommand{\C}{{\mathbb C}}

\newcommand{\N}{{\mathbb N}}
\newcommand{\Q}{{\mathbb Q}}
\newcommand{\GL}{{\operatorname{\bf GL_n}}}

\newcommand{\W}{{\operatorname{\bf W}}}
\newcommand{\w}{{\operatorname{\bf w}}}
\newcommand{\e}{{\mathfrak{e}}}

\newcommand{\rank}{{\operatorname{rank}}}

\newcommand{\pr}{{\operatorname{pr}}}

\newcommand{\de}{{\operatorname{def}}}

\newcommand{\Imm}{\operatorname{Im}}

\newcommand{\sslash}{\mathbin{/\mkern-6mu/}}

\title[Word maps\dots]{Word maps, word maps with constants and representation varieties of one-relator groups}

\author[Gordeev, Kunyavski\u\i , Plotkin] {Nikolai Gordeev, Boris Kunyavski\u\i , Eugene Plotkin}

\dedicatory{To Efim Zelmanov on the occasion of his 60th birthday}

\address{Gordeev: Department of Mathematics, Herzen State
Pedagogical University, 48 Moika Embankment, 191186, St.Petersburg,
RUSSIA} \email{nickgordeev@mail.ru}

\address{Kunyavski\u\i : Department of
Mathematics, Bar-Ilan University, 5290002 Ramat Gan, ISRAEL}
\email{kunyav@macs.biu.ac.il}

\address{Plotkin: Department of Mathematics, Bar-Ilan University, 5290002 Ramat Gan,
ISRAEL} \email{plotkin@macs.biu.ac.il}

\renewcommand{\GL}{\mathrm{GL}}
\newcommand{\SL}{\mathrm{SL}}

\newcommand{\Ma}{\mathrm{M}}

\newcommand{\Res}{\mathrm{Res}}

\newcommand{\PSL}{\mathrm{PSL}}
\newcommand{\PGL}{\mathrm{PGL}}

\def\<{\langle}
\def\>{\rangle}
\def\tilde{\widetilde}

\def\phi{\varphi}

\def \Res{\operatorname{Res}}
\def \W{\mathcal W}
\def \T{\mathcal T}

\def\e{\mathfrak{e}}
\def\m{\mathfrak{m}}
\def\n{\mathfrak{n}}

\def\tr{\operatorname{tr}}

\def\w{\tilde{w}}
\def\we{\tilde{w}^*}

\def\al{\alpha}
\def\be{\beta}
\def\gam{\gamma}
\def\de{\delta}

\def\la{\lambda}

\begin{document}

\begin{abstract}
We consider word maps and word maps with constants on a simple algebraic group $G$. We present results on the images of
such maps, in particular, we prove a theorem on the dominance of “general” word maps with constants,
which can be viewed as an analogue of a well-known theorem of Borel on the dominance of genuine word
maps. Besides, we establish a relationship between the existence of unipotents in the image of the map induced by $w\in F_m$ and
the structure of the representation variety $R(\Gamma_w, G)$ of the group $\Gamma_w = F_m/\left<w\right>$.

\end{abstract}

\numberwithin{equation}{section}

\parindent0pt

\maketitle
\section*{Introduction}

{\it Word maps.} Let $F_m$ be the free group of rank $m$. Fix its
generators $x_1,\dots ,x_m$. Then for any word $w=w(x_1,\dots
,x_m)\in F_m$ and any group $G$ one can define the {\it word map}
$$\w\colon G^m\rightarrow G$$
by evaluation. Namely, $\w(g_1, \dots, g_m)$ is obtained by
substituting $g_i$ in place of $x_i$ and $g_i^{-1}$ in place of
$x_i^{-1}$ followed by computing the resulting value $w(g_1, \dots,
g_m)$.

Word maps have been intensely studied over at least two past decades
in various contexts (see, e.g., \cite{Se}, \cite{Shal}, \cite{BGK},
\cite{KBKP} for surveys). In this paper, we consider the case where
$G=\mathcal G(K)$ is the group of $K$-points of a simple linear
algebraic group $\mathcal G$ defined over an algebraically closed
field $K$. We are mainly interested in studying the image of $\w$.
Borel's theorem \cite{Bo1} says that $\w$ is dominant, i.e., its
image contains a Zariski dense open subset of $G$. However, $\w$ may
not be surjective: this may happen in the case of power maps on
groups with non-trivial centre (say, squaring map on $\SL(2,\mathbb
C)$) and, if $\mathcal G$ is not of type {\sc A}, even on adjoint
groups, see \cite{Ch1}, \cite{Ch2}, \cite{Stei}. For the adjoint
groups of type {\sc A}, the surjectivity problem is wide open, even
in the case of groups of rank 1, and even for words in two
variables.


The goal of the present paper is two-fold. First, we extend our
viewpoint on the dominance and surjectivity problems from genuine
word maps to word maps with constants and establish a partial,
``generic'' analogue of Borel's dominance theorem. Another extension
concerns a continuation of the word map $\w\colon
\GL_n(K)^m\rightarrow \GL_n(K)$ to the map $\w^* \colon
\Ma_n(K)^m \rightarrow \Ma_n(K)$. Being interesting in its own right,
this method yields, as a by-product, a new proof of some results of
Bandman and Zarhin \cite{BZ}, who proved the surjectivity of $\w$
for $G= \PGL_2(K)$ in the case where $K$ is an algebraically closed
field of characteristic zero, $m=2$, and $w\in F_m\setminus F_m^2$,
where $F^1_m = [F_m, F_m], \dots, F_m^i = [F_m^{i-1}, F_m^{i-1}],
\dots$.

Our second goal consists in studying the geometric structure of the
representation variety of the one-relator group
$\Gamma_w:=F_m/\left<w\right>$ with an eye towards applying the data
on its irreducible components to searching unipotent elements in the
image of the word map. This often allows one to prove the
surjectivity of the word map on $\PGL_2(K)$. We give a non-trivial
example of such a $w\in F_2$, in the spirit of \cite{BZ} but
avoiding their computer calculations.

\bigskip

{\it Word maps with constants.} Let $G$ be a group, let
$$
\Sigma = \{\sigma_1, \dots, \sigma_r\,\,\,\mid\,\,\,\sigma_i\in
G\setminus Z(G)\,\,\,\text{for every}\,\,\,i = 1, \dots, r\},
$$
and let $w_1, \dots, w_{r+1} \in F_m$ be reduced words. The expression
$$
w_\Sigma = w_1\sigma_1w_2\sigma_2\cdots w_r\sigma_rw_{r+1}
$$
is called {\it a word with constants} (or a generalized monomial) if
the sequence $w_2, \dots, w_r$ does not contain the identity word.
We will view a word $w\in F_m$ as a word with constants $w_\Sigma$
with $\Sigma =\emptyset$ and $w=w_1$.

A word with constants also induces a map
$$
\w_\Sigma \colon G^m \rightarrow G.
$$

Let now $G=\mathcal G(K)$ where $\mathcal G$ is a simple algebraic
group defined over an algebraically closed field $K$. In general, the image
$\Imm_{\w_\Sigma}$ is not dense in $G$ as in Borel's Theorem.
However, there are examples when this image is dense. For instance,
the problem of the density of $\Imm_{\w_\Sigma}$ is related to the
definition of {\it covering numbers} for products of conjugacy
classes (see \cite{G1}). Namely, let $r = m$ and let
$$
w_\Sigma = x_1 \sigma_1x_1^{-1} x_2 \sigma_2 x_2^{-1}\cdots x_m \sigma_m x_m^{-1}.
$$
Then $\Imm_{\w_\Sigma} = C_1C_2\cdots C_m$ where $C_i$ is the
conjugacy class of $\sigma_i$. Thus,
$$
\overline{\Imm_{\w_\Sigma}} = \overline{ C_1C_2\cdots C_m}
$$
where $\overline{X}$ is the Zariski closure of $X\subset G$. In
\cite{G1} it has been proved that $\overline{\Imm_{\w_\Sigma}} = G$
if $\mid \Sigma\mid >  2\,\rank\, G +1$.

In the present paper, we prove (Theorem \ref{th1}) that for a
``general'' word with constants $w_\Sigma$ the induced map
$\w_\Sigma$ turns out to be dominant.

Note that words with constants are also related to other problems in
the theory of algebraic groups (see, e.g., \cite{G2}, \cite{KT}, \cite{Ste1},
\cite{Ste2}).

\bigskip




{\it Word maps and representation varieties.} 

For a simple algebraic
group $\mathcal G$ defined over an algebraically closed field $K$
one can define the quotient map $\pi \colon \mathcal G\rightarrow
\mathcal G\sslash \mathcal G\thickapprox T/W$ where $\mathcal
G\sslash \mathcal G$ and $T/W$ are categorical quotients with respect to the
action of $\mathcal G$ on $\mathcal G$ by conjugation and the natural action of  the Weyl group $W$ on  a maximal torus $T$,
respectively, see \cite{SS} for details.
Then we have the map
$$
\pi\circ \w\colon \mathcal G^m\rightarrow T/W.
$$


We denote by $\w_K\colon
\mathcal G(K)^m\to \mathcal G(K)$ the induced word map on the group
$G:=\mathcal G(K)$. Borel's theorem implies that $\Imm (\pi\circ
\w_K)$ is dense in $T/W$. 
However, we do not know when $\Imm (\pi\circ \w_K) = T(K)/W$.
Moreover, {\it we have no example} when
$\Imm (\pi\circ \w_K)\ne T(K)/W$.
On the other hand, we have not so many examples when $\Imm (\pi\circ
\w_K) = T(K)/W$. The latter equality holds, say, for the Engel words
$w =[\cdots [[x,y],y],\cdots y]$ (see \cite{G3}) and for the power
words $w = x^k$. Bandman and Zarhin proved in \cite{BZ} that $\Imm
(\pi\circ \w_K) = T(K)/W$ when $G = \SL_2(K)$.


Note that the equality $\Imm (\pi\circ
\w_K) = T(K)/W$ implies that for every semisimple element $s\in
T(K)$ one can find an element of the form $su$ in the image of
$\w_K$ where $u$ is a unipotent element of $G$ which commutes with
$s$, see \cite{SS}. This implies, in its turn, that
$$\Imm (\pi\circ \w_K) =
T(K)/W\Rightarrow \text{ all regular semisimple elements of
}\,\,G \text{ belong to} \Imm \w_K.
$$
Thus, if  $\mathcal G = \SL_2$  then every element of $G$ belongs to
$\Imm \w_K$ except, possibly, $-1, \pm u$ where $u$ is a unipotent
element. Hence, if $\mathcal G = \PGL_2$ then $\Imm \w_K \supseteq
G \setminus \{u\}$. Then we have only one obstacle to proving the
surjectivity of $\w_K$ for $\PGL_2(K)$, namely, we have to prove the
existence of a non-trivial unipotent element in the image. The
existence of unipotent elements in $\Imm \w_K$ is somehow related to
the structure of the representation variety. Namely,  denote
$$
\W_w = \w^{-1}(1),\,\,\,\T_w = (\pi\circ \w)^{-1}(\pi(1))
$$
(here we denote by $1$ the identity of $G$). Then $\W_w$ is the
representation variety $R(\Gamma, G)$ where $\Gamma=\Gamma_w$ is the
$m$-generated group with one defining relation $w$ (see, e.g.,
\cite{LM}). The set $\T_w\subset G^m$ is the set of $m$-tuples
$(g_1, \dots, g_m)\in G^m$ such that
$$
w(g_1, \dots, g_n) = u\,\,\,\text{is a unipotent element of
}\,\,\,G,
$$
that is, $\T_w$ is the preimage of the unipotent subvariety of $G$.
We have an inclusion $\W_w\subseteq \T_w$ of affine subsets of
$G^m$, and the inequality $\W_w\ne \T_w$ is a sufficient condition
for the existence of a non-trivial unipotent element in $\Imm \w$.
We calculate several examples of words in $F_2$ for the group
$\mathcal G = \SL_2$. In all these examples  $\W_w\ne \T_w$ but
there are cases when some irreducible components of $\W_w$ coincide
with a irreducible component of $\T_w$. Possibly, the investigation
of structure properties of the representation varieties $\W_w$ would
give an answer to the question on the existence a non-trivial
unipotent element in $\SL_2(\C)$ lying in the image of $\w$.


\bigskip

Some results of this paper were announced in \cite{GKP}.

\bigskip

{\it Notation and conventions.}

\medskip

Below, if not stated otherwise, $K$ is an algebraically closed field
and $\mathcal G$ is an algebraic group
defined over $K$, so we {\it identify the group $\mathcal G$ with
$G:=\mathcal G(K)$.}

\bigskip
We denote the identity element of $G$ by $1$;

$N_G(H)$ denotes the normalizer of $H$ in $G$;

$R\ast Q$ denotes the free product of groups $R$ and $Q$;

for a group $\Delta$ and $x, y \in \Delta$, we use the
symbol $x \backsim y$ if $x$ is conjugate to $y$ in $\Delta$;


$\mathbb G_{\text{\rm{a}}}$, $\GL_n$, $\SL_n$ denote the additive,
general linear, special linear groups;

$\Ma_n(K)$ is the set of $n\times n$-matrices over $K$; 

$I_n\in \Ma_n(K)$ is the identity matrix;

for $A\in \Ma_n(K)$ by $A^*$ we denote the adjugate matrix, i.e., the
matrix such that $AA^* = A^*A = \det(A) I_n$ (note that for a
generic matrix $M = (x_{ij}) $ the entries of the  matrix $M^*$ are
homogeneous polynomials in $\{x_{ij}\}$ of degree $n-1$);

for a map $f\colon X\to Y$ and a subset $S\subset X$ we denote by
${\text{\rm{Res}}}_Sf$ the restriction of $f$ to $S$.

\bigskip

Let $X$ be an algebraic variety defined over $K$, and let $\{X_i\}$
be a countable set of proper closed subsets $X_i\subsetneqq X$. Then
we call the set $X\setminus \cup_{i}X_i$ {\it a quasi-open subset of
} $X$. (In topology and real analysis such sets are often called
``$G_{\delta}$-sets''.)

\bigskip


If $X$ is an algebraic variety and $Y\subset X$, then $\overline{Y}$ is the Zariski closure of $Y$ in $X$.

\section{Word maps with constants}

Let $\w_\Sigma \colon G^m \rightarrow G$ be a word map with
constants of a simple algebraic group $G$. Note that there are words
with constants $w_\Sigma$ such that $\Imm \w_\Sigma  = 1$ (so-called
identities with constants, see, e.g., \cite{G2}). Such identities
exist if and only if the root system of $G$ contains roots of
different length. However, even in the cases when all roots are of
the same length we cannot expect the analogue of the Borel Theorem
$\overline{\Imm \w_\Sigma} = G $. It would be interesting to
understand the influence of the properties of the set $\Sigma$ on
the dimension of $\overline{\Imm w_\Sigma}$. In such a generality,
this seems to be a difficult question, and we start with considering
some particular situations.


\begin{theorem} \label{pr1}
Let $w_1, \dots, w_{r+1}\in F_m$ be words where $w_2, \dots, w_r \ne
1$. There exists an open set $\mathcal U\subset G^r$ such that for
every $\Sigma = (\sigma_1, \ldots, \sigma_r)\in \mathcal U$ and for
the word with constants $w_\Sigma = w_1\sigma_1w_2\sigma_2\cdots
w_r\sigma_rw_{r+1}$,  the dimension $\dim \overline{\Imm \w_\Sigma}$
is a fixed number $\mathfrak{d} = \mathfrak{d}(w_1, \dots, w_{r+1})$
depending only on $w_1, \dots, w_{r+1}\in F_m$. Moreover,
$$\dim \overline{\Imm \w_{\Sigma^\prime}}\leq \mathfrak{d}$$
for every $\Sigma^\prime = (\sigma^\prime_1, \dots, \sigma^\prime_r)
\in G^r$ (here we admit the possibility $\sigma^\prime_i \in Z(G)$).
\end{theorem}

\begin{proof}
Define the word
$$
w^Y  = w_1y_1w_2y_2\cdots w_ry_rw_{r+1}\in F_{m+r} = \langle x_1,
\dots, x_m, y_1, \dots, y_r\rangle.
$$
We have dominant maps
$$
\w^Y\colon G^{m+r}\rightarrow G\,\,\,\,\text{and}\,\,\,p_Y\colon
G^{m+r}\rightarrow G^r
$$
where $p_Y$ is the projection onto the components $m+1, \dots, m+r$.
Consider the map
$$
\Phi\colon G^{m+r}\stackrel{(\w^Y, p_Y)}{\rightarrow} G\times G^r.
$$
Let $X =\overline{\Imm \Phi}\subset G\times G^r$, and let
$p^\prime_Y\colon X \rightarrow G^r$ be the projection onto $G^r$.
Then $p_Y^\prime(X) = G^r$ because for every $r$-tuple
$\Sigma=(\sigma_1, \dots, \sigma_r) \in G^r$ there is a non-empty
set
$$
Z_\Sigma = \{(\w_\Sigma(g_1, \dots, g_m), \sigma_1, \dots,
\sigma_r)\,\,\,\mid\,\,\, (g_1, \dots, g_m )\in G^m\}\subset X.
$$


One can show that there exists an open subset $\mathcal V$ of $X$
such that:
\begin{itemize}
\item[(a)] $\mathcal V\subset \Imm \Phi$,
\item[(b)] for every $v \in \mathcal V$ the dimension of every
irreducible component of the preimage $\Phi^{-1} (v) $ is a fixed
number $\mathfrak{f}$,
\item[(c)] for every $u \in \Imm \Phi$ the dimension of every
irreducible component of the preimage $\Phi^{-1}(u)$ is greater than
or equal to $\mathfrak{f}$
\end{itemize}
(cf., e.g., \cite[Chapter~I, n.~6, Th.~7]{Shaf}). Let now $\mathcal U\subset
G^r$ be an open subset contained in $p_Y^\prime(\mathcal V)$, and
let $\Sigma = (\sigma_1, \dots, \sigma_r)\in \mathcal U$. Let $v\in
\mathcal V$ be such that $p^\prime_Y(v) = \Sigma$. Then $v =
(\w_\Sigma(g_1, \dots, g_m), \sigma_1, \dots, \sigma_r)$ for some
$(g_1, \ldots, g_m)\in G^m$, and the dimension of every irreducible
component of the preimage $\Phi^{-1} (v) $ is equal to
$\mathfrak{f}$, see (b). Further, the Zariski closure
$\overline{Z}_\Sigma$ is an irreducible closed subset of $X$.
Indeed, $Z_\Sigma$ is the image of an irreducible variety under the
morphism $\Phi_\Sigma\colon G^m\rightarrow G\times G^r$ given by the
formula $\Phi_\Sigma (x_1, \dots, x_m) = (w_\Sigma(x_1, \dots,
x_m), \sigma_1, \dots, \sigma_r)$ (in other words,  $\Phi_\Sigma$ is
the restriction of $\Phi$ to $G^m \times (\sigma_1, \dots, \sigma_r)
\subset G^m\times G^r$). Note that $v \in Z_\Sigma \cap
\mathcal V$. Hence there is an open subset $\mathcal W$ of
$\overline{Z}_\Sigma$ such that $v\in \mathcal W \subset \mathcal
V$. Since $\mathcal W \subset \mathcal V$, the dimension of every
irreducible component of $\Phi^{-1}(v^\prime)$ for every point
$v^\prime \in \mathcal W$ is equal to $\mathfrak{f}$, see (b). Also,
for every $v^\prime \in \mathcal W$ the closed subset
$\Phi^{-1}(v^\prime)\subset G^m\times G^r$ is isomorphic (as an
affine variety) to the closed subset
$\Phi^{-1}_\Sigma(v^\prime)\subset G^m$. Hence the dimension of the
general fibre of the morphism $\Phi_\Sigma \colon G^m\rightarrow
G\times G^r$ is equal to $\mathfrak{f}$, and therefore
$$\dim \overline{\Imm \Phi_\Sigma} = m \dim G - \mathfrak{f}.$$ The
construction of $\Phi_\Sigma$ shows that $\overline{\Imm
\Phi_\Sigma} $ is isomorphic to $\overline{\Imm \w_\Sigma}$ (the
projection of $G\times G^r$ onto the first component gives this
isomorphism). Hence $\dim \overline{\Imm \w_\Sigma} = m \dim G -
\mathfrak{f}$ for every $\Sigma \in \mathcal U$.

\bigskip

Let $\Sigma^\prime = (\sigma^\prime_1, \dots, \sigma^\prime_r) \in
G^r$ (possibly,  $\sigma^\prime_i \in Z(G)$ for some $i$).
Note that over the points of $G \times \Sigma \subset G \times G^r$
the maps
$\w_{\Sigma^\prime} \colon G^m\rightarrow G$ and
$\Phi_{\Sigma^\prime}\colon G^m \rightarrow G\times G^r$ have the
same fibres. Moreover, these fibres are also fibres of the map
$\Phi\colon G^m\times G^r\rightarrow G\times G^r$ which correspond
to points of the form $(\w_{\Sigma^\prime}(g_1, \dots, g_m),
\sigma_1^\prime, \dots, \sigma_r^\prime)$. Since the dimension of
every fibre of $\Phi$ is at least $\mathfrak{f}$ (see (c)), the
dimension $\mathfrak{f}^\prime$ of the general fibre of
$\w_{\Sigma^\prime}$ is at least $\mathfrak{f}$. Hence
$$\dim \overline{\Imm \w_{\Sigma^\prime}} = m\dim G -\mathfrak{f}^\prime \leq \mathfrak{d} = m \dim G-\mathfrak{f}.$$
\end{proof}

\begin{definition}
Given an $(r+1)$-tuple of words $\Omega_r = (w_1, \dots, w_{r+1})$
of $F_m$ where $w_2, \dots, w_r \ne 1$, we say that
$\Sigma=(\sigma_1,\dots ,\sigma_r)$ is {\it regular} for $\Omega_r$
if $\Sigma\in\mathcal U$ where $\mathcal U$ is an open subset
satisfying the conditions of Theorem \ref{pr1}.
\end{definition}

\begin{example}
Let $m =1$ and $\Omega_2 = \{w_1 = x, w_2 = x^{-1}\}$. Then for
every $\Sigma = \{\sigma\}$ the image of $\w_\Sigma$ is the
conjugacy class of $\sigma$, and therefore for a regular $\Sigma$
the dimension of $\Imm \w_\Sigma $ is equal to the dimension of the
conjugacy class of a regular element $\sigma$ of $G$, that is, 
$\dim\, G - \rank\, G$.
\end{example}


\begin{cor}
\label{th1} Let $\Omega_r =(w_1, \dots, w_{r+1})$ be such that $w_2,
\dots, w_r \ne 1$. Suppose
$
\prod_{i=1}^{r+1} w_i \ne 1$.
Then if $\Sigma = (\sigma_1, \dots, \sigma_r)$ is regular for
$\Omega_r$, the map
$
\w_\Sigma \colon G^m \rightarrow G
$
is dominant.
\end{cor}

\begin{proof}
Indeed, for $\Sigma = (1, 1,\dots, 1)$ we have $w_\Sigma = w  =
\prod_{i=1}^{r+1} w_i \ne 1$, and therefore $\w_\Sigma$ is dominant
according to Borel's Theorem. Now the statement immediately follows
from Theorem  \ref{pr1}.
\end{proof}

\begin{remark}
Since the point $(1, 1,\dots, 1)$ is the most peculiar case,
presumably, the condition $w  = \prod_{i=1}^{r+1} w_i \ne 1$ is
sufficient for every $\w_\Sigma$ to be dominant.
\end{remark}



Below we consider one important example where the word with
constants is obtained by substitution of an element of $G$ instead
of one variable and where the condition $\prod_{i=1}^{r+1} w_i \ne1$ of Corollary \ref{th1} may
not hold (such words with constants are used in \cite{G3}).

Let $w (x_1, \dots, x_m, y)\in F_{m+1}$,
$$w = w_1 y^{k_1}w_2y^{k_2}w_3 \cdots w_r y^{k_r}w_{r+1},$$
where $w_1, \dots, w_{r+1} \in F_m$ and $w_2, \dots, w_r \ne 1$.
Then for every $\sigma \in G$ we have the word with constants
$$w_\sigma = w_1 \sigma^{k_1}w_2\sigma^{k_2}w_3 \cdots w_r \sigma^{k_r}w_{r+1}$$
(actually, according to the definition, we have to exclude the cases where $\sigma^{k_i}\in Z(G)$
for some $i$).

\begin{theorem}\label{th2}
Suppose $ \sum_{j=1}^m k_j = 0$.  Then there exists an open subset $\mathcal V\subset G$
such that for any $\sigma \in \mathcal V$ the map
$$\pi \circ \w_\sigma\colon G^m\rightarrow T/W$$
is dominant.
\end{theorem}

\begin{proof}
Choose a sequence $G_1 < G_2 < \cdots <  G_\ell = G$  of simple
algebraic subgroups of $G$ and a sequence $T_1\leq \cdots \leq
T_\ell = T$ of their maximal tori so that $\rank \,G_i = i$ and all
the $G_i$ are $T$-invariant.

This can be done as follows.
Let $\{\alpha_1, \dots, \alpha_\ell\}$ be an irreducible root system
corresponding to $G$. We may assume that the simple roots $\alpha_i$
are numbered so that for every $i \leq \ell$ the set $\{\alpha_1, \dots, \alpha_i\}$
is an irreducible root system. Let $G_i$ be the corresponding subgroup of $G$
(generated by the root subgroups of the root system $\{\alpha_1, \dots, \alpha_i\}$).
Then we have $G_1< G_2 <\cdots <G_\ell = G$.
Further, let
$T_i = \{h_{\alpha_1}(t_1)h_{\alpha_2}(t_2) \cdots h_{\alpha_i}(t_i)\mid \,\,t_1, \dots, t_i\in K^*\}$
be the corresponding maximal torus of $G_i$ (here $\{h_{\alpha_i}(t)\,\mid\, t\in K^*\}$ is the corresponding one-dimensional torus in the simple algebraic group generated by the root subgroups $X_{\pm \alpha_i}$)), and let $T = T_\ell$.
Then we have $T_1< T_2 < \cdots <T_\ell = T$. Obviously, $T$ normalizes every
root subgroup $X_\beta$, $\beta \in \langle \alpha_1, \dots, \alpha_\ell\rangle$,
and therefore it normalizes every $G_i$. Then in the reductive group $G_iT$
we have $T = T_i H_i$ where $H_i$ is the subtorus of $T$ which centralizes $G_i$.
Let $\pi_i\colon G_i \rightarrow
T_i/W_i$ denote the quotient morphisms.



The embedding $\iota\colon  T_{i-1}\rightarrow T_i$ induces a morphism
of varieties
$$\iota^* \colon  T_{i-1}/W_{i-1} \rightarrow T_i/W_i.$$
Obviously, the morphism $\iota^*$ is quasi-finite (every non-empty fibre is finite). 
Then
$$\dim \overline{ \iota^*(T_{i-1}/W_{i-1})} = \dim T_{i-1}= i-1.$$
The set $\overline{\iota^*(T_{i-1}/W_{i-1})}$ is an irreducible codimension one
subset of $T_i/W_i$.

\bigskip

Now suppose that $\sigma \in T$. Recall that for every $i$ we have $T = T_i H_i$
where $H_i$ is the centralizer of $G_i$ in $T$. Hence $\sigma =
\sigma_i h_i$ for some $\sigma_i \in T_i, h_i \in H_i$. Then for
every $m$-tuple $(g_1, \dots, g_m)\in G_i^m$ we have
\begin{equation}
\label{equa1779}
 w_\sigma (g_1, \dots, g_m) = w(g_1, \dots, g_m, \sigma_i)h_i^{\sum_j k_j} = w(g_1, \dots, g_m, \sigma_i)\in G_i
\end{equation}
(recall that $\sum_{j=1}^r k_j = 0$). Hence the restriction of
$\w_\sigma$ to $G_i^m$ gives the map
$$ \w_\sigma^i\colon G_i^m  \rightarrow G_i.$$

\begin{lemma}
\label{lem1} For every $1<i\le\ell$ there is an element $\sigma\in
T$ such that
$$\overline{\pi_i (\Imm \w_\sigma^i)} \nsubseteq \overline{ \iota^*(T_{i-1}/W_{i-1})}.$$
\end{lemma}

\begin{proof}
Assume to the contrary that $\overline{\pi_i (\Imm \w_\sigma^i)}
\subseteq \overline{\iota^*(T_{i-1}/W_{i-1})}$ for every $\sigma \in T$. Then for
every $m$-tuple $(g_1, \dots, g_m)\in G^m_i$ and every $\sigma_i \in
T_i$ we have
$$\pi_i \circ \w (g_1, \dots, g_m, \sigma_i) \subset \overline{ \iota^*(T_{i-1}/W_{i-1})}.$$
Note that any $(m+1)$-tuple of semisimple elements $(\gamma_1,
\dots, \gamma_m, \gamma)$ is conjugate by an element of $G_i$ to an
$(m+1)$-tuple of the form $(g_1, \dots, g_m, \sigma_i)$ where
$\sigma_i \in T$. Since the set of $(m+1)$-tuples $(\gamma_1, \dots,
\gamma_m, \gamma)$ of semisimple elements is dense in $G^{m+1}_i$
and the map $\pi_i \circ \w$ is invariant under conjugation by
elements of $G_i$, we have $\pi_i \circ \w(G^{m+1}_i)\subset
\overline{\iota^*(T_{i-1}/W_{i-1})}$, and therefore
$$\dim \overline{\pi_i \circ \w(G^{m+1}_i}) \leq i-1.$$
This implies that $\overline{\pi_i \circ \w(G^{m+1}_i}) \ne \overline{\pi_i(G_i)}$, 
which is a contradiction with the Borel Theorem \cite{Bo1}.
\end{proof}

\begin{lemma}
\label{lem2} For every $1<i\le\ell$ there is a non-empty open subset
$S_i\subset T$ such that
$$\overline{\pi_i (\Imm \w_\sigma^i)} \nsubseteq
\overline{\iota^*(T_{i-1}/W_{i-1})}
$$
for every  $\sigma \in S_i$.
\end{lemma}

\begin{proof}
Let $\omega^i\colon  G_i^m  \times T \rightarrow G_i$ be given by
the formula $$\omega_i (g_1, \dots, g_m, \sigma) = \w_\sigma ^i(g_1, \dots, g_m).$$
Consider the composite map
$$\pi_i\circ \omega^i\colon G_i ^m\times  T \stackrel{\omega^i}\rightarrow G_i\stackrel{\pi_i}{\rightarrow} T_i/W_i.$$
Let $X = (\pi_i\circ \omega^i)^{-1} (\overline{\iota^*(T_{i-1}/W_{i-1}))}$. Then
$X$ is a proper closed subset in $G_i^m\times T$ by Lemma
\ref{lem1}. Let $X = \cup_q X_q$ be the decomposition of $X$ into
the union of irreducible closed subsets. Further, let $i_T\colon
X\rightarrow T$ be the map induced by the projection $G_i^m\times T
\rightarrow T$. Suppose $\overline{i_T(X_q)}\ne T$ for some $q$. Put
$S_{iq} = T\setminus \overline{i_T(X_q)}$. Now, if $\sigma \in
S_{iq}$ then $i_T^{-1}(\sigma)\cap X_q = \emptyset$. Suppose
$\overline{i_T(X_q)}=  T$ for some $q$.  Since $\dim X_q < \dim\,
G_i^m\times T$, there is a non-empty open subset $S_{iq}\subset T$
such that $\dim i_T^{-1}(\sigma)\cap X_q < \dim G_i^m$ for every $\sigma
\in S_{iq}$. Thus, if $\sigma \in S_i = \cap_q S_{iq}$, the set
$G_i^m\times \{\sigma\}$ is not contained in $X$ and therefore
$\overline{\pi_i (\Imm \w_\sigma^i)} \nsubseteq
\overline{\iota^*(T_{i-1}/W_{i-1}})$.
\end{proof}

\begin{lemma}
\label{lem3} There is a non-empty open subset $S\subset T$ such that
for every $1<i\le\ell$ and every $\sigma\in S$ we have
$$\overline{\pi_i (\Imm\w_\sigma^i)} \nsubseteq \overline{ \iota^*(T_{i-1}/W_{i-1}}).$$
\end{lemma}

\begin{proof}
Take $S = \cap_i S_i$ where $S_i$ is an open set from Lemma
\ref{lem2}.
\end{proof}

We can now prove the theorem. Choose $S$ as in Lemma \ref{lem3}. Let
$\sigma \in S$. Suppose
\begin{equation}
\overline{\pi_{i-1} (\Imm \w^{i-1}_\sigma)} =  T_{i-1}/W_{i-1}.
\label{eq1.8}
\end{equation}
Note that for  $i=1$ we have $T_0 = \{1\}$, $W_0=\{1\}$ and
therefore for $i=1$ assumption \eqref{eq1.8} holds. Since for every $i$ the map  $\w^i_\sigma$ is the restriction of $\w_\sigma$ to $G_i^m$ (see (\ref{equa1779})), we have
$$\pi_i(\Imm \w^{i-1}_\sigma) \subset \pi_i (\Imm\w^i_\sigma).$$
Since $\pi_i(\Imm \w^{i-1}_\sigma)\subset T_i/W_i$ is the image of
$\pi_{i-1}(\Imm \w^{i-1}_\sigma)\subset T_{i-1}/W_{i-1}$ with
respect to the map $\iota^* \colon  T_{i-1}/W_{i-1}\rightarrow T_i/W_i$,
we obtain from \eqref{eq1.8} and the assumption on $\sigma$ that
$$\overline{\iota^*(T_{i-1}/W_{i-1})} \subsetneqq \overline{\pi_i (\Imm \w^i_\sigma)}.$$
Since $\overline{\pi_i (\Imm \w_\sigma^i)}$ and $\overline{\iota^*(T_{i-1}/
W_{i-1})}$ are closed irreducible subsets of $T_i/W_i$ and $\dim \,
\overline{\iota^*(T_{i-1}/ W_{i-1})} $
$= i-1$, we have
$$\overline{\pi_i (\Imm \w_\sigma^i)} =  T_i/ W_i.$$
The induction step now finishes the proof.
\end{proof}

\begin{remark}
Presumably, the assumption $\sum_{j=1}^r k_j = 0$ can be replaced
with a weaker condition: the word $w (x_1, \dots, x_m, y)\in
F_{m+1}$ is not of the form $\omega y^l \omega^{-1}$ for some
$\omega  = \omega(x_1, \dots, x_m, y) \in F_{m+1}$. However, the
example $w(x, \sigma) = x\sigma x^{-1}$ shows that we cannot avoid
some restrictions on the word $w$.
\end{remark}


\begin{remark}
It would be interesting to extend our considerations to words with
constants and automorphisms (anti-automorphisms) of $G$ where we
have variables $x_i$ and variables $x_i^{\phi_i}$ marked by
automorphisms. Such a word also gives rise to a map $G^m\to G$. Note
that given a word with constants $w_\Sigma$ and a collection
$\Phi=\{\phi_i\}$ $(i=1,\dots ,m)$ of automorphisms of $G$, there
are (at least) two natural ways to produce such a map: first replace
each appearance of $x_i^{a_i}$ in $w_\Sigma$ with
$(x_i^{\phi_i})^{a_i}$, then either
\begin{enumerate}
\item replace each such expression with $\phi_i(g_i)^{a_i}$, or
\item do this only in the cases where the exponent $a_i$ is
positive,
\end{enumerate}
and compute the resulting value in $G$. Although the second option
might seem a little artificial, it would include (at least) two
important cases: 1) twisted conjugacy classes, see, e.g., \cite{Sp},
\cite{FT} and the references therein for a survey of various aspects
of this theory, and 2) twisted commutators, see, e.g., \cite{Se},
\cite{NS1}, \cite{NS2}, \cite{Le}.
\end{remark}

\section{Word maps and central functions} 

{\it General definitions.} Let $G = \SL_n(K)$. 
Then we extend a word map with constants $\w_\Sigma\colon
G^m\rightarrow G$ to a map
$$\we_\Sigma\colon \Ma^m_n(K)  \rightarrow \Ma_n(K)$$
in the following way. Let
\begin{equation}
w_\Sigma = w_1\sigma_1w_2\sigma_2\cdots w_r\sigma_rw_{r+1} =
x_{{i_1}}^{a_1}x_{i_2}^{a_2}\cdots x_{i_{s_1}}^{a_{s_1}}\sigma_1
x_{i_{s_1+1}}^{a_{s_1+1}}\cdots x_{i_{s_2}}^{a_{s_2}}\sigma_2
x_{i_{s_2+1}}^{a_{s_2+1}} \cdots \label{eq2.1}
\end{equation}
be a word map with constants where $\Sigma =\{\sigma_1, \dots,
\sigma_r\}\subset \SL_n(K)$.
For any $m$-tuple of matrices $(\mu_1, \dots, \mu_m)\in \Ma^m_n(K)$
we define
$$\we_\Sigma(\mu_1,\dots, \mu_m)\in \Ma_n(K)$$ by putting in formula
\eqref{eq2.1} instead of each expression $x_{i_k}^{a_k}$ the matrix
$\bar{\mu}_{i_k}^{\mid a_k\mid}$ where
$$\bar{\mu}_{i_k} = \begin{cases} \mu_{i_k}\,\,\,\text{if}\,\,\,a_k >0\\ \mu^*_{i_k} \,\,\,\text{if}\,\,\,a_k < 0 \end{cases}.$$
In other words, we substitute in formula \eqref{eq2.1} matrices
instead of variables replacing each negative exponent with its
module and in each such case replacing the given matrix with the
adjugate matrix.

\begin{remark}
The word map with constants $\w_\Sigma\colon \SL_n^m(K)\rightarrow
\SL_n(K)$ admits a natural extension to the group $\GL_n(K)$. Below
we also denote by $\w_\Sigma$ the corresponding map
$\GL_n^m(K)\rightarrow \GL_n(K)$.
\end{remark}

Note that the map $\we_\Sigma$ can be represented by an $n\times
n$-matrix whose entries are polynomial functions in the entries of
the matrices $\mu_i$. More precisely:
\begin{prop}
\label{pr1.1}
\begin{itemize}
\item[(i)]
Let $A= K[\{y_{pq}^r\}^m_{r=1}]$ be the ring of polynomial functions
on $\Ma_n^m$ where $y_{pq}^r$ corresponds to the $(p,q)^{th}$ entry
of the $r^{th}$ component of $\Ma^m_n$ (here $1\leq p,q \leq n,
\,\,1\leq r\leq m$). Further, let
$$a = \sum_{a_k > 0}  a_k,\,\,\,\,b = \sum_{a_k < 0}  \mid a_k\mid .$$
Then
$$
\we_\Sigma = \begin{pmatrix} w_{11}&w_{12}&\cdots &w_{1n}\cr
w_{21}&w_{22}&\cdots &w_{2n}\cr \cdots\cr w_{n1}&w_{2n}&\cdots
&w_{nn}\cr\end{pmatrix}\in \Ma_n(A)
$$
where each $w_{ij}$ is a homogeneous polynomial in $\{y_{pq}^r\}$ of
degree $a + (n-1) b$.
\item[(ii)]
Let $\Delta_r$ be the homogeneous polynomial of degree $n$ in the
variables $y_{pq}^r$ given by the determinant of the $r^{\text{th}}$
component of $\Ma_n$, and let
$$\Delta =
\prod_{r=1}^m\Delta^{b_r}_r\,\,\,\text{where}\,\,\,b_r = \sum_{a_k
<0, i_k = r}\mid a_k\mid.
$$
Then, restricting to $\GL_n$ and $\SL_n$, we obtain, respectively
$$
\Res_{\GL_n}\we_\Sigma = \Delta \w_\Sigma\,\,\,\,\text{and}\,\,\,\Res_{\SL_n}\we_\Sigma = \w_\Sigma.
$$
\end{itemize}
\end{prop}

\begin{proof}
Straightforward from the definition.
\end{proof}

\begin{example}
\label{ex0} Let $m = 1, n =2$,
$$
\Sigma = \left\{\sigma_1 =\sigma = \begin{pmatrix}s&0\cr 0&s^{-1}\end{pmatrix}, \sigma_2 = \sigma^{-1}=\begin{pmatrix}s^{-1}&0\cr 0&s\end{pmatrix}\,\,\,\,\mid\,\,\,s \ne 1\right\},
$$
$w_\Sigma = \sigma x\sigma^{-1}x^{-1}$. Then $
a = 1, b = 1$, and $\we_\Sigma$ is a function on $\Ma_2$ which can be
expressed by the formula
$$
\we_\Sigma = \begin{pmatrix}s&0\cr
0&s^{-1}\end{pmatrix}\begin{pmatrix}y_{11}&y_{12}\cr
y_{21}&y_{22}\end{pmatrix}\begin{pmatrix}s^{-1}&0\cr
0&s\end{pmatrix}\begin{pmatrix}y_{22}&-y_{12}\cr
-y_{21}&y_{11}\end{pmatrix}= $$$$= \begin{pmatrix}y_{11}y_{22} -
s^2y_{12}y_{21}& y_{11}y_{12}( s^2-1)\cr
y_{21}y_{22}(1-s^{-2})&y_{11}y_{22} - s^{-2}y_{12}y_{21} \cr
\end{pmatrix}.
$$
\end{example}


{\it Central functions on $\Imm_{\we_\Sigma}.$} Let $X \in \Ma_n(K)$,
and let
$$
\chi(X) = \lambda^n +\chi_1(X)\lambda^{n-1} +\cdots +\chi_n(X)
$$
be the characteristic  polynomial of $X$. Then the map $\chi_i
\colon \Ma_n(K) \rightarrow K$, $X\mapsto \chi_i(X)$, is given by a
homogeneous polynomial of degree $i$ in the entries of $X$.

\bigskip

The following fact should be compared with \cite[Lemma~2.1]{BZ}.

\begin{theorem}
\label{th11} Let $$\w_\Sigma\colon \SL^m_n(K)\rightarrow \SL_n(K).$$
Then $\chi_i\circ \w_\Sigma\colon \SL_n^m(K)\rightarrow K$ is either
a constant function, or takes every value.
\end{theorem}

\begin{proof} We start with the following lemma.

\begin{lemma}
\label{pr2} Let $F$ be a polynomial function on $\Ma^m_n(K)$ which is
homogeneous on each component of $\Ma_n^m(K)$, that is, there exist
natural $d_i$, $i=1,\dots ,m$, such that
$$
F( g_1, \dots, g_{i-1}, cg_i, g_{i+1}, \dots,  g_m) = c^{d_i}F(
g_1, \dots, g_{i-1}, g_i, g_{i+1}, \dots,  g_m)
$$
for every $(g_1, \dots,  g_m)\in \Ma_n^m(K)$ and every $c \in K$.
Then the restriction of $F$ to $\SL^m_n(K)$ is either a constant
function, or takes every value.
\end{lemma}

\begin{proof}
Suppose that the restriction of $F$ to $\SL^m_n(K)$ is not a
constant function. Let $\alpha \in K$, and denote by $X_\alpha$ the
hypersurface in the affine space $\Ma^m_n(K)$ defined by the equation
$$F^n- \alpha\prod_{r=1}^m \Delta_r^{d_r} = 0.$$ (Here $A = K[\{y_{pq}^r\}_{r=1}^m]$, $\Delta_r$,
$\Delta$ are defined as in Proposition \ref{pr1.1}.) Denote by $X_0$
the hypersurface of $\Ma_n^m(K)$ given by the equation $\prod_{r=1}^m
\Delta_r = 0$. If $X_\alpha\subset X_0$, then
$$F^n- \alpha\prod_{r=1}^m \Delta_r^{d_r}
=\beta\Delta_{i_1}^{e_{i_1}}\Delta_{i_2}^{e_{i_2}}\cdots
\Delta_{i_s}^{e_{i_s}}$$ for some $1 \leq i_1\leq i_2\leq \cdots\leq
i_s \leq m$, $\beta \in K^*$ and $e_{i_j} \in \N$ (note that each
$\Delta_r$ is an irreducible polynomial in a polynomial ring $A$),
and therefore $F^n(g_1, \dots, g_m) = \alpha + \beta$ for every
$(g_1, \dots, g_m) \in \SL^m_n(K)$, which contradicts our
assumption. Hence, there is $(g_1, \dots, g_m) \in \GL^m_n(K)$ such
that $$F^n(g_1, \dots, g_m) = \alpha \prod_{r=1}^m
\Delta_r^{d_r}(g_r).$$  On the other hand, $g_r =
\sqrt[n]{\Delta_r(g_r)}g^\prime_r$ for some $g^\prime_r \in
\SL_n(K)$. Now, using the condition on $F$ we obtain
$$F^n(g_1,
\dots, g_m) = F^n(g^\prime_1, \dots, g^\prime_m)\prod_{r=1}^m
\Delta_r^{d_r}(g_r) .$$ Hence $F^n(g^\prime_1, \dots, g^\prime_m)
= \alpha$. Thus, the function $F^n$ takes every value on
$\SL^m_n(K)$. The same is of course true for the function $F$.
\end{proof}

Let now define $F = \chi_i\circ \we_\Sigma \colon \Ma_n^m(K)
\rightarrow K$. Then $F$ satisfies the assumptions of Lemma
\ref{pr2} (see Proposition \ref{pr1.1}). Hence $\Res_{\SL^m_n} F$ is
either a constant function, or takes every value. We have
$$\Res_{\SL^m_n} F =
\Res_{\SL^m_n}\chi_i\circ \we_\Sigma = \chi_i \circ
\Res_{\SL^m_n}\we_\Sigma  = \chi_i \circ \w_\Sigma,
$$
which proves the theorem.
\end{proof}

\begin{cor}
\label{cor1} If $$\chi_1\circ\w_\Sigma \colon \SL^m_2(K)\rightarrow
\SL_2(K)$$ is not a constant function, then every non-central
semisimple conjugacy class of $\SL_2(K)$ intersects $\Imm
\w_\Sigma$.
\end{cor}

\begin{proof}
Let $C$ be a non-central semisimple conjugacy class of $\SL_2(K)$,
and let $\alpha = \chi_1(g) = \tr g$ for $g \in C$. Further, let
$(g_1, \dots, g_m)\in \SL_2^m(K)$ be such that $\tr (\w_\Sigma
(g_1, \dots, g_m)) = \alpha$. Then $\w_\Sigma (g_1, \dots,
g_m)\in C$.
\end{proof}

This result gives rise to an alternative proof of the following
theorem by Bandman and Zarhin.

\begin{cor} [cf. \cite{BZ}]
\label{cor2} The image of a non-trivial word map $\w\colon
\SL^m_2(K)\rightarrow \SL_2(K)$ contains every semisimple element
except, possibly, $-1$.
\end{cor}

\begin{proof}
We may view the word $w$ as $w_\Sigma$ with $\Sigma = \emptyset$.
Also, the identity element $1$ is always in $\Imm\, w$.
\end{proof}

Here is another corollary.

\begin{theorem} \label{th3}
Let $G$ be a simple algebraic group. Suppose that $G$ is not of type $A_r, r > 1$,
$D_{2k+1}$, $k > 1$, or $E_6$, and let $\w\colon G^m \rightarrow G$
be a non-trivial word map. Then every regular semisimple element of
$G$ is contained in $\Imm \w$. Moreover, for every semisimple $g\in
G$ there exists $g_0\in G$ of order $\leq 2$ such that $gg_0\in
\Imm\,\w$.
\end{theorem}

\begin{proof}
Let $R$ be an irreducible root system of rank $r$. Let us check the
property: there exists a  root subsystem $R^\prime \subset R$ such that
$$
R^\prime = R_1\cup R_2\cup\cdots \cup R_r\subset R \,\,\,\text{where}\,\,\,R_i = A_1.\eqno(*)
$$
Property (*) implies that
$$
G_1\times G_2\times \cdots \times G_r\leq
G_R\,\,\,\text{where}\,\,\,G_i = SL_2(K), PGL_2(K)\eqno(**)
$$

The following fact is apparently well known (one can extract it,
e.g., from \cite[Table 5 on page 234]{GOV}). We present a proof for
the sake of completeness. We use the notation of \cite{Bou}
throughout.

\begin{lemma} \label{lem:roots}
Property (*) holds for all irreducible root systems except for
$A_r$, $r > 1$, $D_{2k+1}$, $k > 1$, and $E_6$.
\end{lemma}

\begin{proof}

{\bf Case $A_r$.}  Obviously, (**) does not hold ($\SL_2^r \nless
\SL_{r+1}$). Then (*) does not hold for $r > 1$.

\bigskip

{\bf Case $B_r$.} We have $R = \langle \e_1 -\e_2, \dots,
\e_{r-1}-\e_r, e_r\rangle$.

\medskip

{\bf Subcase $r = 2m +1$.} We have $R_1 = \langle \e_1 -\e_2\rangle
, R_2 =\langle  \e_1 +\e_2\rangle, R_3 = \langle \e_3 -\e_4\rangle ,
R_4 =\langle \e_3 +\e_4\rangle, \dots, R_{2m} = \langle \e_{2m-1}
+\e_{2m}\rangle, R_{2m+1} = \langle \e_{2m+1}\rangle$.

\medskip

{\bf Subcase $r = 2m$.} We have $R_1 = \langle \e_1 -\e_2\rangle ,
R_2 =\langle  \e_1 +\e_2\rangle, R_3 = \langle \e_3 -\e_4\rangle ,
R_4 =\langle \e_3 +\e_4\rangle, \dots, R_{2m} = \langle  \e_{2m-1}
+\e_{2m}\rangle.$

\bigskip

{\bf Case $C_r$.} We have $R_i = \langle 2\e_i\rangle$ (long roots).

\bigskip

{\bf Case $D_{2m}.$} The same as $B_{2m}$.

\bigskip

{\bf Case $D_{2m+1}$.} Consider the standard representation of
$D_{2m+1}$ with weights $\pm \e_1, \dots, \pm \e_{2m+1}$. For a root
$\alpha = \e_i \pm \e_j$, the semisimple root subgroup $h_\alpha(t)$
acts non-trivially exactly on the four weight vectors corresponding
to $\pm \e_i, \pm \e_j$. Then if (*) holds, we can divide the
dimension of the representation  by $4$. But this dimension equals
$2(2m+1)$, so (*) does not hold for $r = 2m+1$.

\bigskip

{\bf Case $\bf E_6$.} In $E_6$ there are no six mutually orthogonal
positive roots. Indeed, since all roots lie in the same $W$-orbit,
we may start with
$$
\alpha = \frac{1}{2}(\e_1+\e_2 +\e_3+\e_4+\e_5 -\e_6-\e_7+\e_8).
$$
There are altogether fifteen positive roots orthogonal to $\alpha$:
ten roots of the form
$$
\beta = -\e_i +\e_j,\,\,1\le i< j\leq 5,
$$
and five roots of the form
$$
\gamma =  \frac{1}{2}(\pm \e_1 \pm \e_2 \pm \e_3 \pm \e_4 \pm \e_5
-\e_6-\e_7+\e_8)
$$
(where the number of minus signs is equal to $4$). The set of roots
of the form $\beta$ contains orthogonal subsets of size at most two,
and the set of roots of the form $\gamma$ contains no orthogonal
subsets. Hence one can find at most three mutually orthogonal
positive roots, and thus (*) does not hold.

\bigskip

{\bf Case $E_7$.} Here $D_6 \cup \langle \e_7-\e_8\rangle \subset
R$, and therefore (*) holds.

\bigskip

{\bf Case $E_8$.} Here $D_8 \subset E_8$, and therefore (*) holds.

\bigskip

{\bf Case $F_4$.} Here $D_4\subset F_4$, and therefore (*) holds.

\bigskip

{\bf Case $G_2$.} Here $\langle \e_1 -\e_2\rangle \cup \langle -
2\e_3 +\e_1 +\e_2\rangle\subset G_2$, and therefore (*) holds.
\end{proof}

We can now prove the theorem. Let $\Gamma_i = \SL_2(K)$, and let
$$
\Psi\colon \prod_{i=1}^r\Gamma_i \rightarrow G
$$
be the natural homomorphism induced by the inclusion (**).

Denote by $\w_i$ the word map $\Gamma_i^m\rightarrow \Gamma_i$
defined by the same word $w$. Then $\Imm \w_i$ contains every
semisimple element of $\Gamma_i$ except, possibly, $-1$.

Let $T_i$ be a maximal torus of $\Gamma_i$.

We have $\Psi(\prod_{i=1}^r T_i)= T$. (Indeed, $\Psi(\prod_{i=1}^r
T_i)\subset T$, and $\Psi(\prod_{i=1}^r T_i)$ is a torus of
dimension $r = \dim T $.) Let $t \in T$. Then
$$
t = \Psi(t_1,t_2,\dots, t_r)\,\,\,\text{for some}\,\,\,t_i \in T_i.
$$
First suppose that $t_i\ne -1$ for every $i$. Then
$$
t_i = \w_i(g_{i1}, \dots, g_{im})\,\,\,\text{for
some}\,\,\,\,(g_{i1}, \dots, g_{im})\in \Gamma_i^m,
$$
and therefore
$$
t = \w(\gamma_1, \dots, \gamma_m) \,\,\,\text{where}\,\,\,\gamma_j
= \Psi(g_{1j},g_{2j},\dots ,g_{rj}).
$$
Note that the condition $t_i\ne -1$ holds for all regular semisimple
elements $t$ (indeed, if $t_i = -1$, then $\Psi(1, \dots, 1, u_i, 1\dots 1)\ne 1$ commutes
with $t$ for a non-trivial unipotent element $u_i\in
\Gamma_i\thickapprox \SL_2(K)$). Then all such elements lie in
$\Imm\, \w$.

If $t_i  = -1$ for some $ i$, then we may take
$$
t_0 = (t_1^\prime, t_2^\prime, \dots,
t_r^\prime)\,\,\,\text{where}\,\,\,t_i^\prime = \begin{cases} -
1,\,\,\text{if}\,\,\,t_i = -1\\  1\,\,\,\text{if}\,\,\,t_i \ne
-1\end{cases}.
$$
Then the order of $\Psi(t_0)$ is at most two, and $\Psi(tt_0) \in
\Imm\, \w$.
\end{proof}

\begin{remark}
Presumably, Theorem \ref{th11} can be extended in the following way.
Instead of one function $\chi\circ \w$ one could consider the map
$\pi \circ \w$ where $\pi  =(\chi_1, \dots, \chi_{n-1}) \colon
\SL_n\rightarrow T/W = \mathbb A^{n-1}$ is the quotient map. Say,
for a tuple $\alpha = (\alpha_1, \dots, \alpha_{n-1})\in T/W$ one can
consider the system of equations $$[\chi_i \circ \we (Y)]^n
-\alpha_i \prod_{r=1}^m \Delta_r^{d_{ir}}(Y)=0$$ where $d_{ir}$ is
the homogeneous degree of $\chi_i \circ \we (Y)$ with respect to the
variables $y_{ij}^r$ (see Proposition \ref{pr1.1}, Lemma \ref{pr2}).
Obviously, this system has solutions in $\Ma_n^m(K)$. If the variety
of these solutions is not contained in the variety
$\prod_{r=1}^m\Delta_r = 0$, we can find a solution in $\GL_n^m(K)$
of the system $\chi_i \circ \w (Y) -\sqrt[n]{\alpha_i}  = 0$ for
some $Y\in \SL_n^m(K)$.
\end{remark}

\begin{remark}
Note that if we could prove that $\pi\circ \w$ is a surjective map,
we would have every regular semisimple element of any simple group
$G$ in $\Imm \w$. Indeed, in every irreducible root system of rank
$r$ there is a subsystem of rank $r$ which is a union of systems
$A_i$ (see, e.g., \cite{Bo1}).
\end{remark}



\section{Representation varieties and generic groups}

{\it General constructions} (see, e.g., \cite{LM}, \cite{Pl},
\cite{PBK}, \cite{Si}).

Let $\Gamma = \langle g_1, \dots, g_m\rangle $ be a finitely
generated group, and let $\mathfrak{R}_{\Gamma}\subset F_m$ be the
set of all relations of $\Gamma$.


Put
\begin{equation}
R(\Gamma, G) = \{(x_1, \dots, x_m) \in
G^m\,\,\,\mid\,\,\,\omega(x_1, \dots, x_m) = 1\,\,\,\,\text{for
every}\,\,\, \omega \in \mathfrak{R}_{\Gamma}\}.\label{eqn1}
\end{equation}
Obviously, $R(\Gamma, G)$ is a Zariski closed subset of $G^m$ (which
is defined over $K$), and for every $(x_1, \dots, x_m)\in R(\Gamma,
G)$ the subgroup $\langle x_1, \dots, x_m\rangle \leq G$ is a
quotient of $\Gamma$.  One can identify the sets
$$
R(\Gamma, G)  = \operatorname{Hom}(\Gamma, G)
$$
using the one-to-one correspondence
$$
(x_1, \dots, x_m)\in R(\Gamma, G) \leftrightarrow \rho \in
\operatorname{Hom}(\Gamma, G)
$$
given by
$$
(\rho(g_1), \dots, \rho(g_m)) = (x_1,\dots, x_m).
$$
The set $R(\Gamma, \GL_n(K))$ is called {\it the variety of
$n$-dimensional representations of $\Gamma$}.

\bigskip

The ``variety'' $R(\Gamma, G)$ may be non-reduced and reducible, so
the scheme language is most appropriate, see, e.g., \cite{Si}.
However, we will freely use the abusive term ``variety'' in what
follows.

Let $R(\Gamma, G) = \cup_i R(\Gamma, G)^i$ be the decomposition into
the union of irreducible closed subsets. Then we have the following
property for the components.

\begin{prop}
\label{pr51} Let $K$ be an algebraically closed field of infinite
transcendence degree over
a prime subfield.
Then for each $i$ there
exists a dense quasi-open subset $U^i\subset R(\Gamma, G)^i$ such
that for every $(g_1, \dots, g_m)\in U^i$ the subgroup $\langle g_1,
\ldots, g_m\rangle $ is isomorphic to a fixed quotient of $\Gamma$.
\end{prop}

\begin{proof}
For any $\omega\in F_m$ the set
$$
X_\omega = \{(g_1, \dots, g_m) \in G^m\,\,\,\mid\,\,\,\omega(g_1,
\dots, g_m) = 1\}
$$
is a proper closed subset of $G^m$. Hence $R(\Gamma, G)^i \cap
X_\omega$ is a closed subset of $R(\Gamma, G)^i$. Let now
$$
Q_i = \{\omega \in F_m\,\,\,\,\mid\,\,\,R(\Gamma, G)^i\nsubseteq X_\omega\}.
$$
Then
$$\bigcup_{\omega\in Q_i}R(\Gamma, G)^i \cap X_\omega
$$
is a countable union of proper closed subsets of $R(\Gamma, G)^i$.
Since $K$ is of infinite transcendence degree over
a prime subfield,
we have
$$
U^i =R(\Gamma, G)^i \setminus \bigcup_{\omega\in Q_i} (R(\Gamma, G)^i\cap X_\omega) \ne \emptyset
$$
(see \cite{Bo2}), and therefore $U^i$ is a dense quasi-open subset
of $R(\Gamma, G)^i$. Every group $\langle g_1, \dots, g_m\rangle$
for $(g_1, \dots, g_m)\in U^i$ has the same set of relations:
$\{\omega \in F_n\,\,\,\mid\,\,\,R(\Gamma, G)^i\subset X_\omega\}$.
\end{proof}

\begin{definition}
A group isomorphic to $\langle g_1, \dots, g_m\rangle$ for $(g_1,
\dots, g_m)\in U^i$ will be called the {\em generic group} of the
component $R(\Gamma, G)^i$.
\end{definition}

For every $\rho\in R(\Gamma, G)$ and every $g \in G$  the map
$\rho_g\colon \Gamma \rightarrow G$ given by $\rho_g(\gamma) =
g\rho(\gamma)g^{-1}$ is also an element of $R(\Gamma, G)$. Hence we
have a regular action of the algebraic group $G$ on the affine set
$R(\Gamma, G)$. If $G = \GL_n(K)$ or $\SL_n(K)$, orbits correspond
to classes of equivalent representations. In these cases (or, more
generally, if $G$ is a reductive group) there exists a categorical
quotient 
$R(\Gamma, G)\sslash G$ which is also a closed affine set.
There is a one-to-one correspondence between points of $R(\Gamma,
\GL_n(K))\sslash \GL_n(K)$  (resp.
$R(\Gamma,\SL_n(K))\sslash\SL_n(K)) $ and classes of completely
reducible $n$-dimensional representations of $\Gamma$ (see
\cite{LM}).

\bigskip

Let $\mathfrak{B}_\Gamma\subset\mathfrak{R}_\Gamma$ be a minimal set
of relations (that is, $\mathfrak{B}_\Gamma$ is a minimal set of
generators of the group $\mathfrak{R}_\Gamma$ as a normal subgroup
of $F_m$). Since the equality $\omega (x_1, \dots, x_m) = 1$, where
$x_1, \dots, x_m \in G, \omega \in F_m$, implies that $\omega
(yx_1y^{-1}, \dots, yx_my^{-1}) = 1$ for every $y \in G$, we may
reduce the set of all relations $\mathfrak{R}_\Gamma$ in
\eqref{eqn1} to the set $\mathfrak{B}_\Gamma$.  Let us now assume
that $\mathfrak{B}_\Gamma = \{\omega_1, \dots, \omega_k\}$ is a
finite set. Then we may consider the map
$$
\Phi_{\mathfrak{B}_\Gamma}\colon G^m\rightarrow G^k
$$
defined by
$$
\Phi_{\mathfrak{B}_\Gamma} (x_1,\dots, x_m) = (\omega_1(x_1,\dots,
x_m), \dots, \omega_k (x_1,\dots, x_m)).
$$
Then
\begin{equation}
R(\Gamma, G) = \Phi_{\mathfrak{B}_\Gamma}^{-1}(e_k)\label{eqn2}
\end{equation}
(here $e_k = (1, \dots ,1)$ is the identity of $G^k$), see
\cite{LM}.

\bigskip

{\it  Finitely generated one-relator groups.} Let
$\mathfrak{B}_\Gamma = \{w\}$.  Then we will write $\Gamma_w$
instead of $\Gamma$ to emphasize the relation $w$.

Further, here $k=1$ and the map $\Phi_{\mathfrak{B}_\Gamma}\colon
G^m\rightarrow G$ is the word map
$$
\tilde{w}\colon G^m\rightarrow G.
$$
In this case we denote
$$
\tilde{w}^{-1}(1) = \W_w.
$$

From \eqref{eqn2} we have $\W_w = R(\Gamma_w, G).$ Denote by
$\W_w^i$ the irreducible components of $\W_w$ where $i = 0, 1,
\dots$.

\bigskip

In what follows, the image of the identity element $1$ of a fixed
torus $T$ of $G$ in the quotient variety $T/W$ is also denoted by
$1$. Define
$$\T_w = (\pi\circ\tilde{w})^{-1} (1)$$
and denote by $\T_w^j$ the irreducible components of $\T_w$ where $j
= 0, 1, \dots$. Then the following simple statement is of key
importance.

\begin{prop}\label{pr3333}
$\, $
\begin{itemize}
\item[(a)] For every irreducible component $\W_w^i$ of $\W_w$ there is an
irreducible component $\T_w^j$ of $\T_w$ which contains $\W_w^i$.
\item[(b)] If $\W_w \ne \T_w$, then $\Imm \, \w$ contains a non-trivial unipotent element $u$.
\end{itemize}
\end{prop}

\begin{proof}

$\,$

(a) Obviously, $\W_w\subset \T_w$ and hence every irreducible component of $\W_w$ is contained in an irreducible component of  $\T_w$.

(b) Indeed, the set $\T_w\subset G^m$ is exactly the set of $m$-tuples
$(g_1, \dots, g_m)\in G^m$ such that $w(g_1, \dots, g_m)$ is a
unipotent element.
\end{proof}

\begin{remark} \label{rem111}
Presumably, in (b) one can replace ``if'' with ``if and only if''.
\end{remark}

If $K$ is 
a field of infinite transcendence degree over a prime subfield, then
for each $i$ there exists a dense quasi-open subset $U^i\subset
\W_w^i(K)$ such that for every $(g_1, \dots, g_m)\in U^i$ the
subgroup $\langle g_1, \dots, g_m\rangle $ is isomorphic to a fixed
quotient of $\Gamma_w$ (Proposition \ref{pr51}).

The generic group  $\langle g_1, \dots, g_m\rangle $, where $(g_1,
\dots, g_m)\in U^i$, will be denoted by $\tilde{\Gamma}_w^i$. The
question on describing the possibilities for $\tilde{\Gamma}_w^i$
for a given group $\Gamma_w$ is interesting in its own right. More
specifically, answering it could help in describing words $w$ with
the condition $\W_w \ne \T_w$, which guarantees the existence of
non-trivial unipotent elements in $\Imm\,\w$. Below we consider some
examples when $\W_w \ne \T_w$.

\section{Examples of $R(\Gamma_w, \SL_2(\C))$}

In this section we use the following notation:

$G = \SL_2(\C)$;

$T, B, B^-, U, U^-$ are the sets of diagonal, upper and lower
triangular, and upper and lower unitriangular matrices;

$\dot w_0$ is an element of $N_G(T)\setminus T$;

$w \in F_2$;

$R(\Gamma_w, \SL_2(\C)) = \mathcal W_w = \bigcup_{i=1}^l \W^i_w$ is
the decomposition into the union of irreducible components;

$\T_w$ is the hypersurface in $G\times G$ defined by the equation
$\tr w(x,y)=2$;

$\T_w=\bigcup_{j=1}^e \T_w^j$ is the decomposition into the union of
irreducible components.

Obviously, for every $j$ we have
\begin{equation}
\dim \T_w^j = 5. \label{eqq1}
\end{equation}
Also, for every $i = 1, \dots, l$ there is $j = 1, \dots, e$ such
that $\W_w^i \subset T_w^j$ and
\begin{equation}
3\le\dim \W_w^i \le 5. \label{eqq2}
\end{equation}
(Indeed, the upper inequality follows from \eqref{eqq1}. Since
$\W_w$ is defined by three equations $w_{11} =1, w_{12}=w_{21} = 0$,
where
$$
\tilde{w}(x, y) = \begin{pmatrix}w_{11}(x, y)&w_{12}(x, y)\cr
w_{21}(x, y)&w_{22}(x, y)\cr \end{pmatrix}\,\,\,\mid\,\,x, y \in
G\},
$$
the dimension of every component of $\W_w$ is at least $\dim
(G\times G) - 3 = 6-3 = 3.$)

\bigskip

Note that according to Proposition \ref{pr3333},
$$\dim \W_w^i \leq 4 \Rightarrow \text{there exists a non-trivial unipotent element }\,\, u \in \Imm\,\w.\eqno(*)$$

We start with some simple examples.

\begin{example} \label{ex1} $\bf w = [x,y]$.
\end{example}
In this case $\W_w$ is classically known as the {\it commuting
variety} of $G$. See, e.g., \cite{RBKC} and the references therein
for its properties and related problems.

\begin{prop}
\label{pr37} The set $\W_w$ is irreducible, $\dim \W_w = 4$ and
\begin{equation}
\W_w = \overline{\{g (T\times T) g^{-1}\,\,\,\mid\,\,\,g \in
G\}}.\label{eq3.1}
\end{equation}

The set $\T_w$ is also irreducible, $\dim \T_w = 5$, and
$$
\T_w = \overline{\{g (B\times B) g^{-1}\,\,\,\mid\,\,\,g \in
G\}}.\label{eq3.2}
$$
\end{prop}

\begin{proof}
The irreducibility of the commuting variety in $G\times G$ has been
proven in \cite{Ri} in a more general case where $G$ is a semisimple
simply connected group. A general pair in $\W_w$ is a pair of
commuting semisimple elements. Hence we have \eqref{eq3.1}. Since the general $G$-orbit (under conjugation) of a pair $(t_1, t_2)$ is of dimension $\dim G/T =2$, we have
$\dim \, \W_w = 4$. 

\bigskip

Let $(g_1,g_2)\in \T_w^j$ be a general pair. We may assume $\pm 1\ne g_2 = t \in T$
(since $\dim \T_w^j=5$, it cannot happen that all pairs $(g_1, g_2)$
consist of $\pm$ unipotent elements). Show that $g_1 \in B\cup
B^{-}$. Assume the contrary. Then either $g_1 = v s u$ for some
$1\ne v \in U^{-}, 1\ne u \in U, s \in T$, or $g_1 = \dot w_0 u$ for
some $\dot w_0$ and $u \in U$. In the first case,
$$
[vsu, t] = vsutu^{-1}s^{-1}v^{-1}t^{-1} = vu^\prime v^\prime
\,\,\,\text{for some}\,\,\,v^{-1}\ne v^\prime \in U^-,\,\,1\ne
u^\prime\in U.
$$
Indeed, since both $s$ and $t$ are diagonal matrices in $\SL_2$ and
$t \ne \pm 1$, we have $txt^{-1} \ne x$ for every $x \in U$ or $x
\in U^-$, hence
$$vsutu^{-1}s^{-1}v^{-1}t^{-1} = v\underbrace{s\underbrace
{u (tu^{-1}t^{-1})}_{\in U, \ne 1} s^{-1}}_{:= u^\prime\in U, \ne 1}
\underbrace{(tv^{-1}t^{-1})}_{:=v^\prime \in U^-, \ne v^{-1}}=
vu^\prime v^\prime .$$
Hence $[vsu,t] \backsim v^{\prime\prime}
u^\prime$ for some $1\ne v^{\prime\prime} \in U^-, 1\ne u^\prime\in
U$. But an element of the form $v^{\prime\prime} u^\prime$ cannot
have the trace equal to $2$. This is a contradiction.

In the second case, we have $g_1 = \dot w_0 u$. Then $\tr [\dot w_0
u, t] = \tr t^{-2} \ne 2$. This is also a contradiction.

Thus, we have a pair $(g_1, g_2)$ where $\pm 1 \ne g_1\in T$ and
$g_2 \in B\cup B^-$. Hence $(g_1, g_2) \in \{g (B\times B) g^{-1}\,\,\,g \in G\}$.
This implies that
$$
\T_w^j \subset \overline{\{g (B\times B)
g^{-1}\,\,\,\mid\,\,\,g \in G\}}.
$$
The opposite inclusion $\T_w^j \supset \overline{\{g (B\times B)
g^{-1}\,\,\,\mid\,\,\,g \in G\}}$ follows from the equality  $[B, B]
= U$. \end{proof}



\begin{example} \label{ex2}
$\bf w = [x^m, y^n]$.
\end{example}

Denote $$\m  = \{k\,\,\,\in \N\,\mid\,\,\, 2< k \mid\,
2m\},\,\,\,\,\n = \{l \in \N\,\mid\,\,\,  2 < l \mid\,2n\}.$$ Let
$C_r \subset G$  be the conjugacy class of elements of order $r$.
\begin{prop}
\label{pr39} We have
$$\W_w^0 = \W_{[x, y]} \subset \T_w^0 = \T_{[x, y]}.$$
All other irreducible components are of one of the following forms:
$$\W_w^j = \T_w^j = C_j \times G$$
for $j \in \m $ or
$$\W_w^j = \T_w^j =  G\times C_j$$
for $j \in \n$.
\end{prop}
\begin{proof}
The existence of the components $\W_w^0 = \W_{[x, y]} \subset \T_w^0
$ is obvious.

Let now $(x, y) \in \W_w^j $ be such that $x^m\ne \pm 1, y^n \ne \pm
1$. Then the elements $x, x^m$ are either in the same torus of $G$
or in the same unipotent subgroup (modulo the centre). The same is
true for $y, y^n$. Since $[x^m, y^n] =1$ we then have $[x,y] =1$ and
therefore $(x, y) \in \W_w^0$. Thus, if $(x, y) \in \W_w^j\setminus
\W_w^0$ then either $x^m =\pm 1$ or $y^n = \pm 1$ and $\W_w^j$ is
one of the components for $j\in\m $ or $j \in \n$ which have been
pointed out in the statement. Since the components $\W_w^j$, $j>1,$
are isomorphic to the direct product of a conjugacy class $C_g$ for
a semisimple element $g$ of order $>2$ and the group $G$, we have
$\dim \W_w^j = 5$ and therefore $\W_w^j = \T_w^j$.
\end{proof}

\begin{definition}
\label{def40} We say that a subgroup $H\leq G$ is a free product
modulo the centre if $H/Z(G) \approx R\ast Q$ for some $R, Q\leq
G/Z(G)$. In this case we write
$$H = R\ast Q \mod Z(G)$$
\end{definition}

Now for the integer $j$ put
$$[j] = \begin{cases} j\,\,\,\text{if}\,\,\, j\,\,\,\text{is odd},\\
\frac{j}{2}\,\,\,\text{if}\,\,\, j\,\,\,\text{is even}.\end{cases}$$

\begin{prop}
\label{pr41} For $j>0$ we have $$\tilde{\Gamma}_w^j = \mathbb
Z_{[j]}\ast \mathbb Z \mod Z(G)\,\,\,\,\text{or}\,\,\,\,
\tilde{\Gamma}_w^j = \mathbb Z\ast \mathbb Z_{[j]} \mod Z(G)$$ where
$\mathbb Z$ is the infinite cyclic  group and $\mathbb Z_{[j]}$
is the cyclic group of order $[j]$.
\end{prop}

\begin{proof}
For any two non-central conjugacy classes $C_1, C_2 \subset G$, the
generic group $\langle g_1, g_2\,\,\mid\,\, g_1 \in C_1, g_2 \in
C_2\rangle$ is isomorphic to $\langle g_1\rangle \ast \langle g_2
\rangle \mod Z(G)$, see \cite{G2}. Hence the statement follows from
Proposition \ref{pr39}.
\end{proof}

\begin{example} \label{ex3}
$\bf w = [x,y]^2. $
\end{example}

\bigskip

Let
$$t = \begin{pmatrix} \la &0\cr
0&\la^{-1}\cr \end{pmatrix}, g =  \begin{pmatrix} \al&\be\cr \gam&\de\cr \end{pmatrix}\in G.$$
We have
\begin{equation}
[t, g] = \begin{pmatrix}\al\de-\be\gam \la^2& &\al\be(-1+\la^2)\cr
&&\cr
 \gam\de(-1+\la^{-2})& &\al\de - \be\gam\la^{-2}\cr
\end{pmatrix}.\label{equa135}
\end{equation}
Hence
$$\tr ([t, g])  = 2 - \be\gam( \la
-\la^{-1})^2 .
$$
Thus,
\begin{equation}
\tr ([t, g]) = a
\Leftrightarrow 2 -\be\gam( \la
-\la^{-1})^2 = a \Leftrightarrow  \be\gam = \frac{2-a}{( \la
-\la^{-1})^2}: = p_{a, \la}.\label{equa131}
\end{equation}
Since $\det g = 1$, we have
\begin{equation}
\be\gam = \frac{2-a}{( \la
-\la^{-1})^2}\Leftrightarrow \al\de = 1 + \frac{2-a}{( \la
+\la^{-1})^2} =   \frac{ \la^2
+\la^{-2} -a}{(\la
-\la^{-1})^2} := q_{a, \la}. \label{equa132}
\end{equation}
Put
\begin{equation}
T_a:=\{t = \begin{pmatrix} \la &0\cr
0&\la^{-1}\cr \end{pmatrix}, \,\,\,\la \ne \pm 1, \la^2
+\la^{-2} -a \ne 0\}.\label{equa127}
\end{equation}
If $t \in T_a$, $a\ne 2$, then
 \begin{equation}
 M^a_t = \left\{\begin{pmatrix} \al&\be\cr \gam&\de\cr
\end{pmatrix}\in G\,\,\,\mid\,\,\,\be\gam = p_{a, \la} \right\}=
\left\{\begin{pmatrix} \al&\be\cr \gam&\de\cr
\end{pmatrix}\,\,\,\mid\,\,\,\be\gam = p_{a, \la}, \al\de = q_{a, \la} \right\}\,\,\,\label{equa130}
\end{equation}
is an irreducible closed subset in $G\times G$ and $\dim M^a_t = 2$.
The construction of $M^a_t$ implies that
\begin{equation} \label{equa21}
M_T^a:  = \{(t, M^a_t) \mid t\in T_a\} = \{(t, g)\,\,\mid\,\,\,t\in T_a, g\in G, \tr ( [t,g])  = a\}.
\end{equation}
Also, the set $M^a_T$  is an irreducible locally closed subset of $G\times G$,
and $\dim M_T^a = 3$.  Now let $\Psi\colon  M_T^a\times G\rightarrow G\times G$ be defined by $\Psi (t,g,y) = (yty^{-1}, ygy^{-1})$. Since $M^a_T$  is an irreducible locally closed subset of $G\times G$ and $G$ is an affine variety, the closure of the image of $\Psi$ is an irreducible closed subset of $G$. Thus, the set
\begin{equation}
S_a : = \overline{\Imm\, \Psi} = \overline{\{g M_T^a g^{-1}\,\,\mid\,\,g \in G\}}  = \overline{ \{(gtg^{-1}, gM^a_tg^{-1}) \mid t\in T_a, g\in G\} }\label{equa22}
\end{equation}
is an irreducible closed subset of $G\times G$. Further,  the projection $p\colon S_a \rightarrow G$ onto the first component of $G\times G$ is dominant
because the image is invariant under conjugation and contains  every $t \in T_a$. The fibre $p^{-1}(t)$ is equal to $M_t^a$ and therefore is of  dimension $2$. Hence
\begin{equation}
\dim S_a = 5.\label{equa23}
\end{equation}

Note that we also have \eqref{equa23} and the irreducibility of $S_2$ by Proposition  \ref{pr37}.

We need the following irreducibility statement, which
is probably known to experts. Following the referee's suggestion, we provide
a self-contained proof.

\begin{lemma}
\label{lem44}
Let $a\in \mathbb C$. Then the set $\{(g_1, g_2) \,\,\,\mid\,\,\,\tr ([g_1, g_2]) = a\}$ is irreducible and
$$ \{(g_1, g_2) \,\,\,\mid\,\,\,\tr ([g_1, g_2]) = a\} = S_a.$$

\end{lemma}

\begin{proof} The irreducible closed subset $S_a$  is contained in $ \{(g_1, g_2) \,\,\,\mid\,\,\,\tr ([g_1, g_2]) = a\}$ (see \eqref{equa21}, \eqref{equa22}). Equality \eqref{equa23} implies that $S_a$ is an irreducible component of  the set $ \{(g_1, g_2) \,\,\,\mid\,\,\,\tr ([g_1, g_2]) = a\}$.
Suppose that there exists an irreducible component $S_a^1 \ne S_a$ of  the set $ \{(g_1, g_2) \,\,\,\mid\,\,\,\tr ([g_1, g_2]) = a\}$.
 Since the set $ \{(g_1, g_2) \,\,\,\mid\,\,\,\tr ([g_1, g_2]) = a\}$ is a hypersurface in $\SL_2(\C)\times \SL_2(\C)$, all its irreducible components are of dimension $5$. Thus,
\begin{equation}
\dim S_a^1 = 5.\label{equa24}
\end{equation}
Let $p_1\colon  S_a^1\rightarrow G$ be the projection onto the first component of $G\times G$. Since the set $S_a^1$ is invariant under conjugation by elements of $G$ and it is an irreducible closed  subset of $G \times G$, the map $p_1$ is either dominant or its image is contained  in a single conjugacy class $C$. In the latter case
we have $S_a^1 = C\times G$  and $C \ne \pm 1$ (this follows from \eqref{equa24}). However, one can find pairs $(g_1, g_2), (g_3, g_4)\in C\times G$ such that
$[g_1, g_2] = 1, [g_3, g_4] \ne 1$ and therefore $\tr([g_1, g_2]) = 2, \tr([g_3, g_4]) \ne 2$. Thus, the set $S_a^1$ cannot be of the form $ C\times G$. Hence the map
$p_1$ is dominant and there exists an open subset $T_a^1\subset T_a$ such that $T_a^1\subset \Imm p_1$. Now for every $t\in T_a^1$ we have
$$S_a^1 \supset \{gp^{-1}_1(t)g^{-1}\,\,\,\mid\,\, t\in T^1_a, g \in G\}\subset  \{g M_T^a g^{-1}\,\,\mid\,\,g \in G\}\subset S_a.$$
Since $T_a^1$ is an open subset of the torus $T$, the set $X = \{gp^{-1}_1(t)g^{-1}\,\,\,\mid\,\, t\in T^1_a, g \in G\}$ contains an open subset of the component $S_a^1$. But $X$ is also a subset of $S_a^1$.
Thus, $S_a = S_a^1$.
\end{proof}

\bigskip

Now we consider our case $w = [x, y]^2$.

Since the condition $\tr([x, y]^2 ) = 2$  implies that $[x, y] = \pm u$ where $u$ is a unipotent element, we have
$\T_w = S_{2}\cup S_{-2}$. The sets $S_{\pm 2}$ are irreducible of dimension $5$. Thus,  $S_2 = \T_w^0, S_{-2} = \T_w^1$. Note that the set $\T_w^0$
is the variety $\T_{[x, y]}$ considered in Example \ref{ex1}. Now consider the set  $\T_w^1$. The definition $T_a$ (see \eqref{equa127}) implies that
$$T_{-2} = \left\{t = \begin{pmatrix} \la &0\cr
0&\la^{-1}\cr \end{pmatrix}, \,\,\,\la \ne \pm 1, \la^2
+\la^{-2} +2 \ne 0\right\} = \left\{t = \begin{pmatrix} \la &0\cr
0&\la^{-1}\cr \end{pmatrix}, \,\,\,\la \ne \pm 1, \pm i\right\} .$$
Then \eqref{equa22} implies
$$\T_w^1 =   \overline{ \{(gtg^{-1}, gM^{-2}_tg^{-1})\,\, \mid \,\,t\in T,\,t^4\ne 1, g\in G\} }$$
where $M_t^{-2}$ is defined by \eqref{equa130}. Let $t_0 = \begin{pmatrix}i&0\cr 0&-i\end{pmatrix}$, it is an element of order $4$. Then we can also define $M_{t_0}^{-2}$ by formula \eqref{equa130} where $p_{a, \la} = -1$ and $q_{a, \la} = 0$ (see \eqref{equa131}, \eqref{equa132}), namely,
$$M_{t_0}^{-2} = \left\{\begin{pmatrix} \al&\be\cr \gam&\de\cr
\end{pmatrix}\,\,\,\mid\,\,\,\be\gam = -1, \al\de = 0\right\}.$$
Definitions \eqref{equa131}, \eqref{equa132}, \eqref{equa130} show that $M_{t_0}^{-2} = \{g \in G\,\,\mid\,\,\tr([t, g]) = -2\}$. Hence the set $(t_0, M_{t_0}^{-2})$ is a subset of $S_{-2}$ and therefore we can rewrite the formula for $\T_w^1$:
$$\T_w^1 = \overline{\{gt g^{-1}\times gM^{-2}_t g^{-1}\,\,\,\mid\,\,\,t \in T, \,t^2\ne 1, \, g \in G\}}.$$

\begin{prop}

Each of the sets $\T_w$ and $\W_w$ has two irreducible components:
$$\W_w^0 = \W_{[x, y]} \subset \T_w^0 = \T_{[x, y]},\,\,\,\W_w^1 \subset \T_w^1,$$
where
$$\W_w^1 = \left\{g \left(\begin{pmatrix}i&0\cr 0&-i\cr \end{pmatrix},
\begin{pmatrix}0&\mu\cr -\mu^{-1}&0\cr \end{pmatrix}\right)  g^{-1}\,\,\,\mid\,\,\, \mu\in\C^*,\,\,\, g \in G\right\},\,\,\,\dim \W_w^1 = 3.$$
The generic group $\tilde{\Gamma}_w^1$ of $\W_w^1$ is the quaternion group $Q_8$.
\end{prop}

\begin{proof}
We have only one component $\W_w^0  = \W_{[x, y]}\subset \T_w^0 = \T_w$ (see Example \ref{ex1}).

Consider a pair $(g_1, g_2) \in \T_w^1$ such that
$ [g_1, g_2] =- 1$. Then $g_1 \ne \pm u$ where $u$ is a unipotent element. Indeed, if $g_1 =\pm u$ we may assume
$1\ne u \in U$ and $g_2  \in \dot w_0 U$. Let $g_2 = \dot w_0 u^\prime, v\in U$. Then $$[g_1, g_2] = (\pm u)\dot w_0 u^\prime (\pm u^{-1})u^{\prime-1}\dot w_0^{-1} = u \underbrace{\dot w_0(u^\prime u^{-1}u^{\prime-1})\dot w_0^{-1}}_{= v \in U^-, v\ne 1} = uv \ne \pm 1.$$


Thus, we may assume
$$ g_1 = \begin{pmatrix}\la&0\cr 0&\la^{-1}\cr \end{pmatrix},\,\,\,g_2 =
\begin{pmatrix} \al&\be\cr \gam&\de\cr \end{pmatrix}$$
where $\la \ne \pm 1$. Then formula \eqref{equa135} shows that the
equality $[g_1, g_2] = -1$ is possible if and only if
\begin{equation}
g_1 =  \pm \begin{pmatrix}i&0\cr 0&-i\cr \end{pmatrix},\,\,\,g_2 =
\begin{pmatrix} 0&\mu\cr -\mu^{-1}&0\cr \end{pmatrix}.\label{equa136}
\end{equation}
In these cases $\langle g_1, g_2\rangle = Q_8.$  Since every pair $(g_1, g_2)$  with the property $[g_1, g_2] = -1$  is conjugate to a pair of the form \eqref{equa136}, we have only one irreducible component
\begin{equation}
\W_w^1 = \left\{g \left(\begin{pmatrix}i&0\cr 0&-i\cr \end{pmatrix},
\begin{pmatrix} 0&\mu\cr -\mu^{-1}&0\cr \end{pmatrix}\right)g^{-1}\,\,\,\mid\,\,\,g \in G\right\}.\label{equa137}
\end{equation}
If $C_{t_0}$ is the conjugacy class of $t_0$ then $\W_w^1$ is the hypersurface in $C_{t_0}\times C_{t_0}$ given by the equation $[x, y] = -1$.
Hence $\dim \W_w^1 = 2\dim C_{t_0} - 1 = 3$.
Also equality \eqref{equa137} shows
$$(g_1, g_2) \in W_w^1 \Leftrightarrow \langle g_1, g_2\rangle =
Q_8.$$
\end{proof}

\begin{example} \label{ex4}
$\bf w = [ x,y]^p,\,\,\, p\ne2  $, where $p$ is a prime number.
\end{example}

\begin{prop}
There are $1 + \frac{p-1}{2}$ irreducible components $\T_w^j$ and
the same number of components $\W_w^j\subset \T_w^j$. Moreover,
$$ \W_w^0 = \W_{[x,y]} \subset \T_{[x,y]} = \T_w^0$$
and for $j = 1, 2, \dots ,\frac{p-1}{2}$ we have
$$\W_w^j = \T_w^j = \{(x, y) \,\,\,\,\mid\,\,\,\tr ([x, y]) = 2\cos \frac{{2}j\pi}{p}\}.$$
\end{prop}

\begin{proof}
Obviously, we have
$$ \W_w^0 = \W_{[x,y]} \subset \T_{[x,y]} = \T_w^0.$$
Let now $\T_w^j\ne \T_w^0$. Then
$$(x, y) \in \T_w^j \setminus \T_w^0\Leftrightarrow \begin{cases} \tr ([x, y]) \ne 2 \\ \text{and}\\ [x, y]^p = 1.\end{cases}
\Leftrightarrow [x, y] \backsim\begin{pmatrix}\epsilon_p&0\cr
0&\epsilon_p^{-1}\cr \end{pmatrix}\,\,\,\text{where}\,\,\epsilon_p =
\sqrt[p]{1}\ne 1.$$ The condition $[x, y]
\backsim\begin{pmatrix}\epsilon_p&0\cr 0&\epsilon_p^{-1}\cr
\end{pmatrix}$ is equivalent to $\tr([x, y]) = 2\cos \frac{2j\pi}{p}$
for some $j = 1, \dots, \frac{p-1}{2}$.
Lemma  \ref{lem44} shows that there are exactly $\frac{p-1}{2}$ irreducible components $\T_w^j = \{(x, y)\,\,\,\mid\,\,\,\tr([x,
y]) = 2\cos \frac{2j\pi}{p}\} = S_{2\cos \frac{2j\pi}{p}}$ apart from the component $\T_w^0$. Also, for $j>0$ we have
$$\W_w^j =\{(x, y) \in \T_w^j\,\,\,\mid\,\,\,[x, y]^p = 1\} = \T_w^j.$$
\end{proof}

\begin{remark}
The same arguments as in the case $[x, y]^2 = 1$ show that for $j>0$
we have
$$\T_w^j=\overline{\{gt g^{-1}\times gM^{2\cos \frac{2j\pi}{p}}_t g^{-1}\,\,\,\mid\,\,\,t \in T, t^2\ne 1,  g \in G\}}$$
where  $$M^{2\cos \frac{2j\pi}{p}}_t = \left\{\begin{pmatrix} \al&\be\cr \gam&\de\cr \end{pmatrix}\in G\,\,\,\mid\,\,\,\be\gam =  \frac{2(1 - \cos\frac{2\pi j}{p})}{(\la-\la^{-1})^2}\right\}.$$
\end{remark}

\begin{prop}
For $j>0$ all $\tilde{\Gamma}_w^j$ are non-solvable infinite groups
isomorphic to each other.
\end{prop}

\begin{proof}
A solvable group $\Gamma=\langle g_1, g_2\rangle $ in $\SL_2(\C)$
with the relation $[x, y] \backsim\begin{pmatrix}\epsilon_p&0\cr
0&\epsilon_p^{-1}\cr \end{pmatrix}$ can only be a generalized
quaternion group $Q_{4p}$ (this follows from the classification of
subgroups of $\SL_2(\C)$). However, a finite group may only have a
finite number of non-equivalent representation in $\SL_2(\C)$. Hence
if $\Gamma$ is the generic group of some irreducible component, then
$\Gamma$ is conjugate in $\SL_2(\C)$ to a fixed subgroup $Q_{4p}\leq
\SL_2(\C)$. It is easy to see that the dimension of such a component
is then equal to $3$, and therefore such a component is a proper
closed subset of some $\W_w^j$. Hence the generic group
$\Gamma=\tilde{\Gamma}_w^j$ of the component $\W_w^j$ cannot be
solvable.

\bigskip

Let $F\subset \C$ be a pure transcendental extension of $\Q$ of
infinite degree, and let $\overline{F}\subset \C$ be its algebraic
closure. There exists $\sigma\in \operatorname{Gal}(\overline{F}/F)$
such that $\sigma(\epsilon_p) = \epsilon_p^2$. Then
$\{\sigma^r(\epsilon_p)\}$ where $r = 1, \dots ,p-1$, is the set of
all roots $\sqrt[p]{1}\ne 1$. Further, the set $\W_w$ is
$\Q$-defined and therefore $F$-defined. Since $\overline{F}$ is an
algebraically closed field of infinite transcendence degree over
$\Q$, one can find a pair $(g_1, g_2) \in (\SL_2(\overline{F})\times
\SL_2(\overline{F})) \cap \W_w^1$ such that $\langle g_1, g_2\rangle
= \tilde{\Gamma}_w^1$ (see Proposition \ref{pr51}). Further, the set
$\W_w$ is $ \operatorname{Gal}(\overline{F}/F)$-stable. Thus
\begin{equation}
(\sigma^r(g_1), \sigma^r(g_2))  \in \W_w^j(\overline{F}) =
(\SL_2(\overline{F})\times \SL_2(\overline{F})) \cap
\W_w^j\,\,\,\,\text{for some}\,\,\,j = j(r).\label{equa3.5}
\end{equation}
Further,
\begin{equation}
[g_1, g_2] \backsim \begin{pmatrix}\epsilon_p&0\cr
0&\epsilon_p^{-1}\cr \end{pmatrix}\Rightarrow [\sigma^r(g_1),
\sigma^r(g_2)] \backsim \begin{pmatrix}\epsilon_p^{2^r}&0\cr
0&\epsilon_p^{-2^r}\cr \end{pmatrix}.\label{equa3.6}
\end{equation}
Then \eqref{equa3.5} and \eqref{equa3.6} imply that for every $j =
1, \dots , \frac{p-1}{2}$ there exists $r$ such that
$(\sigma^r(g_1), \sigma^r(g_2))\in \W_w^j(\overline{F})$.
Since $\sigma^r$ is an automorphism of $\SL_2(\overline{F})$, we get
$$\tilde{\Gamma}_w^j  = \langle \sigma^{r(j)}(g_1),
\sigma^{r(j)}(g_2)\rangle \approx \tilde{\Gamma}_w^1.$$
\end{proof}

\bigskip

Let us now consider a more complicated example.

\begin{example} \label{ex5}
$\bf w (x, y) = [ [x, y], x[x,y]x^{-1}]$.
\end{example}
We have $w\in F^2_2=[[F_2,F_2],[F_2,F_2]].$

We remind that by $w_0\in W$ we denote the non-trivial element of
the Weyl group. Note that all elements of the form $\dot
w_0=\begin{pmatrix} 0&\alpha\cr -\alpha^{-1}&0\cr \end{pmatrix}\in
N_G(T)$ whose image in $W$ is $w_0$ belong to the same conjugacy
class. We denote this class by $C_{\dot w_0}$ (and often shorten $\dot
w_0$ to $\dot w$).

\begin{theorem}
\label{pr77} Each of the varieties $\T_w$ and $\W_w$ has three
irreducible components:
$$ \T_w^0 = \T_{[x, y]}, \T_w^1 =\overline{\{g (T\times \dot w B) g^{-1}\,\,\,\mid\,\,\,g \in G\}},\T_w^2 = C_{\dot w}\times G,$$
$$\W_w^0= \T_w^0,\W_w^1 = \overline{\{g (T\times \dot w T) g^{-1}\,\,\,\mid\,\,\,g
\in G\}}\subset \T_w^1,\,\,\,\dim \W_w^1 = 4,\W_w^2 = \T_w^2.$$
\end{theorem}

\begin{proof}
If $(g_1, g_2)\in \T_{[x, y]}$, then $g_1$ and $g_2$ are in the same
Borel subgroup, and therefore $w(g_1, g_2) = 1$. Hence the
$5$-dimensional irreducible variety $\T_{[x,y]}$ coincides with an
irreducible component of $\T_w$ and $\W_w$. Thus we may put
$$\W_w^0 = \T^0_w = \T_{[x,y]}.$$

We start with the following lemma.

\begin{lemma}
\label{lem78} Let $s \in T, s^4 \ne 1$, and let $h \in \dot w
B$. Then $w(s, h) = v \in U^-$ and
$$v = 1 \Leftrightarrow h = \dot w.$$
\end{lemma}

\begin{proof}
We have $h = \dot w u$ for some $\dot w$ and some  $u\in U $ and
$$w(s, h) =[ [s, \dot w u], s [ s, \dot w u] s^{-1}].$$
Further,
$$[s, \dot w u] = s\dot w  u s^{-1} u^{-1}\dot w^{-1} =( s\dot w s^{-1})
\underbrace{(s u s^{-1} u^{-1})}_{ = u^\prime \in U}\dot w^{-1}  =
\underbrace{( s\dot w s^{-1}\dot w^{-1})}_{= s^2} \underbrace{\dot w
u^\prime \dot w^{-1}}_{:= v^\prime \in U^-} = s^2 v^\prime$$ and
$v^\prime =  1\Leftrightarrow h =\dot w$. Then
$$s [ s, \dot w u] s^{-1} = s^2 \underbrace{(s v^\prime s^{-1})}_{:= v^{\prime\prime}\in U^-} = s^2 v^{\prime\prime} $$
and $v^{\prime\prime} = v^\prime  \Leftrightarrow v^\prime = 1
\Leftrightarrow h =\dot w$. Then
$$w(s, h) = [s^2 v^\prime, s^2 v^{\prime\prime}] = s^2 v^\prime
s^2 \underbrace{v^{\prime\prime} v^{\prime -1}}_{:= v_1 \in  U^-}
s^{-2} v^{\prime\prime -1} s^{-2} =  s^2 v^\prime \underbrace{s^2
v_1s^{-2}}_{:= v_2\in U^-} v^{\prime\prime -1} s^{-2} =$$
$$ =  s^2 v^\prime v_2  v^{\prime\prime -1} s^{-2} = s^2 v_2
\underbrace{v^\prime  v^{\prime\prime -1}}_{= v_1^{-1}} s^{-2} = s^2
\underbrace{v_2v_1^{-1}}_{:=v_3}s^{-2}  =  s^2 v_3s^{-2}:= v$$ and
$v = 1\Leftrightarrow v_3 = 1  \Leftrightarrow v_2 = v_1
\Leftrightarrow v^{\prime\prime} =  v^\prime\Leftrightarrow h =
\dot w$.
\end{proof}

Lemma \ref{lem78} implies that $\tr (w(s, h)) =2$ for every
$s\in T$ (if $s^4 = 1$ then $s^2 = \pm 1$ and $w(s, h) = 1$ for
every $h \in \dot w B$). Then the set
$$\overline{\{g (T\times \dot w B) g^{-1}\,\,\,\mid\,\,\,g \in G\}}$$
is contained in an irreducible component of $\T_w$. It is easy to
see that
$$\dim \overline{\{g (T\times \dot w B) g^{-1}\,\,\,\mid\,\,\,g \in G\}} = \dim T + \dim B + \dim G/T = 1 + 2 + 2 = 5,$$
and therefore the set
$$\overline{\{g (T\times \dot w B) g^{-1}\,\,\,\mid\,\,\,g \in G\}}:= \T_w^1$$
is an irreducible component of $\T_w$. Further, Lemma \ref{lem78}
implies that a general pair $(g_1, g_2)\in T\times \dot w B$
satisfies the condition $w(g_1, g_2) = 1$ if and only if $g =\dot w$
for some $\dot w$. Hence
$$\W_w^1 :=\overline{\{g (T\times \dot w T) g^{-1}\,\,\,\mid\,\,\,g \in G\}}$$
is the only irreducible component of $\W_w$ contained in $\T_w^1$,
and
$$\dim \W_w^1 =\dim T +\dim T +\dim G/T = 4. $$

\bigskip

Let $(g_1, g_2) \in C_{\dot w}\times G$. Since $g_1^2 = -1$, we have
$$ g_1[g_1,g_2]g_1^{-1} = g_1(g_1g_2g_1^{-1}g_2^{-1}) g_1^{-1} = - g_2g_1^{-1}g_2^{-1}g_1^{-1} =
- g_2(-g_1)g_2^{-1}g_1^{-1} = [g_2, g_1] = [ g_1, g_2]^{-1},$$ and
therefore
$$w(g_1,g_2) = [ [g_1, g_2], g_1[g_1,g_2]g_1^{-1}] = [[g_1, g_2],  [g_1, g_2]^{-1}] = 1.$$
Hence the $5$-dimensional variety $C_{\dot w}\times G$ coincides
with an irreducible component of $\W_w$ and also with an irreducible
component of $\T_w$. Hence we may put
$$\W_w^2 = \T_w^2 = C_{\dot w}\times G.$$

\bigskip

To prove that neither $\T_w$, nor $\W_w$ contain additional
irreducible components, we need several computational lemmas.

\begin{lemma}
\label{lem79} Let $g_1 = \begin{pmatrix} \la&0\cr 0&\la^{-1}\cr
\end{pmatrix}$, $\la \ne \pm 1$, and let
$$g_2 = \begin{pmatrix} \al&\be\cr \gam&\de\cr \end{pmatrix}.$$
Then
$$[g_1, g_2] \,\,\,\text{is a unipotent element if and only if}\,\,\, g_2 \in B\cup B^-$$
and
$$\begin{cases} [g_1, g_2]\in B\Leftrightarrow \gam = 0\,\,\,\text{or}\,\,\,\de = 0
\Leftrightarrow g_2 \in B\,\,\,\text{or}\,\,\,g_2 \in B\dot w,\\
[g_1, g_2]\in B^-\Leftrightarrow \be = 0\,\,\,\text{or}\,\,\,\al =
0\Leftrightarrow g_2 \in B^-\,\,\,\text{or}\,\,\,g_2 \in \dot
wB\end{cases}.$$
\end{lemma}

\begin{proof}
We have (see \eqref{equa135})
$$[g_1, g_2] = \begin{pmatrix}\al\de- \be\gam \la^2& &\al\be(-1+\la^2)\cr
&&\cr
 \gam\de(-1+\la^{-2})& &\al\de - \be\gam\la^{-2}\cr
\end{pmatrix}$$
and $\tr ([g_1, g_2]) = 2 -\be\gam( \la -\la^{-1})^2 $.   Therefore, $$\tr
([g_1, g_2]) = 2 \Leftrightarrow \be = 0\,\,\,\text{or}\,\,\,\gam = 0
\Leftrightarrow g_2 \in B\,\,\,\text{or}\,\,\,g_2 \in B^-.$$

Further,
$$[g_1, g_2]\in B\Leftrightarrow \gam\de(-1+\la^{-2}) = 0 \Leftrightarrow \gam = 0\,\,\,\text{or}\,\,\,\de = 0 \Leftrightarrow g_2 \in B\,\,\,\text{or}\,\,\,g_2 \in B\dot w.$$
The case $[g_1, g_2]\in B^-$ can be treated by the same arguments.
\end{proof}

\begin{lemma}
\label{lem101} Let $u$ be a non-trivial unipotent element of $G$,
and let $C_{\pm u}$ be the conjugacy class of $\pm u$. Then $\T_w^j
\ne C_{\pm u}\times G$.
\end{lemma}

\begin{proof}
Obviously, we may consider the case $C_u$ where $u =
\begin{pmatrix}1&1\cr 0&1\cr \end{pmatrix}$. Take
$$ g = \begin{pmatrix}  i\frac{\sqrt{2}}{2}&-i \frac{\sqrt{2}}{2}\cr -i\sqrt{2}&0\cr\end{pmatrix}.$$
Then
$$z = [u, g] = \begin{pmatrix}1&1\cr 0&1\cr \end{pmatrix}\begin{pmatrix}  i\frac{\sqrt{2}}{2}&-i \frac{\sqrt{2}}{2}\cr -i\sqrt{2}&0\cr\end{pmatrix}
\begin{pmatrix}1&-1\cr 0&1\cr \end{pmatrix} \begin{pmatrix}  0&i \frac{\sqrt{2}}{2}\cr i\sqrt{2}& i\frac{\sqrt{2}}{2}\cr\end{pmatrix} = $$$$= \begin{pmatrix}  -i\frac{\sqrt{2}}{2}&-i \frac{\sqrt{2}}{2}\cr -i\sqrt{2}&0\cr\end{pmatrix}
 \begin{pmatrix}  -i\sqrt{2}&0\cr i\sqrt{2}& i\frac{\sqrt{2}}{2}\cr\end{pmatrix} = \begin{pmatrix}0&\frac{1}{2}\cr
- 2&0\cr \end{pmatrix}.$$ Further,
$$w(u, g) = [ z, u z u^{-1}]  =( z u z u^{-1}) (z^{-1} uz^{-1} u^{-1}) =  ( z u z u^{-1})((- z) u( -z ) u^{-1})$$$$=
( z u z u^{-1})^2 = \left(\begin{pmatrix}0&\frac{1}{2}\cr -
2&0\cr \end{pmatrix}\begin{pmatrix}1&1\cr 0&1\cr
\end{pmatrix}\begin{pmatrix}0&\frac{1}{2}\cr -2&0\cr
\end{pmatrix}\begin{pmatrix}1&-1\cr 0&1\cr \end{pmatrix}\right)^2 =$$
$$= \left(\begin{pmatrix}0&\frac{1}{2}\cr
-2&-2\cr \end{pmatrix}\begin{pmatrix}0&\frac{1}{2}\cr -2&2\cr
\end{pmatrix}\right)^2 = \begin{pmatrix}-1&1\cr 4&-5\cr \end{pmatrix}^2.$$
The latter matrix is not unipotent because its trace $\ne 2$. Thus
we get a contradiction with our assumption.
\end{proof}

\begin{lemma}
\label{lem81} Let $\T_w^j \ne \T_w^0, \T_w^1, \T_w^2$ be an
irreducible component of $\T_w$, and let $(g_1, g_2)\in \T_w^j$ be a
general pair. Then $g_1$  is a semisimple element of  order $\ne 1,
2, 4$, $[g_1, g_2]$ is a semisimple element, and $[g_1, g_2] \ne \pm
1$.
\end{lemma}

\begin{proof}
Suppose that the projection of $\T_w^j \subset G\times G$ onto the
first component is contained in a single conjugacy class $C$ of $G$.
Since $\T_w^j$ is a closed subset invariant under $G$-conjugation
and $\dim T_w^j = 5$, we have $\T_w^j = C\times G$. Lemma
\ref{lem101} implies that $C$ is a semisimple class of order $\ne 1,
2$. Also, it cannot be of order $4$ because this would imply $\T_w^j
= \T_w^1$. So the order of $g_1$ is not equal to $1, 2, 4$.

Suppose that $\T_w^j \ne  C\times G$.

Let $\pr_1 \colon G\times G\rightarrow G$ be the projection onto the
first component. Then the set $\pr_1(\T_w^j)$ is not contained in a
single conjugacy class. Since $\overline{\tr (\pr_1(\T_w^j))}$ is an
irreducible closed subset of $K$, we have $\overline{\tr
(\pr_1(\T_w^j))} = K$, and therefore for a general pair $(g_1, g_2)
\in \T_w^j$, $g_1$ is a semisimple element of infinite order.

\bigskip

We may now assume $g_1 = t \in T$ and $t^4 \ne 1$.  Let $(t, g) \in
\T^j_w$ be a general pair. Suppose $[t,g]$ is a unipotent element.
Then $g \in B$ or $g \in B^-$ by Lemma \ref{lem79}. Hence $(t, g)
\in T\times B$ or $(t, g) \in T\times B^-$, and therefore
$$(t, g) \in \{h (B\times B )h^{-1}\,\,\,\mid\,\,\, h \in G\} \subset  \T_{[x, y]}^0.$$
This contradicts the assumption $\T_w^j \ne \T_w^0$. Thus $[g_1,
g_2]$ is a semisimple element (if $[g_1,
g_2] = [t, g] = - u$ where $u\ne 1$ is a unipotent element, then $-u$ and  $-tu t^{-1}$ belong to different Borel  subgroups by (\ref{equa131}), (\ref{equa132}), and therefore $(g_1, g_2) = (t, g)\notin \T_w$ by Lemma \ref{lem79}).

\bigskip

Suppose now $[ t,g] =  1$.  Then the general pair $(t, g)$ belongs
to $ T\times T$ and therefore $(t, g) \in \W_{[x,y]}^0 \subset
\T_w^0$, once again contradicting the assumption $\T_w^j \ne
\T_w^0$.

Suppose  $[ t,g] =  -1$. We have (see \eqref{equa135})
$$[t, g] = \begin{pmatrix}\al\de- \be\gam\la^2& &\al\be(-1+\la^2)\cr
&&\cr
 \gam\de(-1+\la^{-2})& &\al\de - \be\gam\la^{-2}\cr
\end{pmatrix} = \begin{pmatrix}-1&0\cr 0&-1\cr \end{pmatrix}\Rightarrow$$
$$ t =  \pm \begin{pmatrix}i&0\cr 0&-i\cr \end{pmatrix},\,\,\,g = \begin{pmatrix} 0&\mu\cr -\mu^{-1}&0\cr \end{pmatrix}.$$
This is a contradiction with the choice of $t$.
\end{proof}

\begin{lemma}
\label{lem85} Let $g\in G$ be a semisimple element of order
different from $1, 2, 4$, let $h\in G$, and suppose that $s =
[g,h]$ is a semisimple element $\ne \pm 1$. If $s$ and
$gsg^{-1}$ are in the same Borel subgroup, then $s$ and $g$ are also
in the same Borel subgroup.
\end{lemma}

\begin{proof}
We may assume $s\in T$ and $s, gsg^{-1}\in B$. If $g\notin B$, then
$g = b_1\dot w b_2$  for some $b_1, b_2\in B$ and
$$g s g^{-1}\in B\Leftrightarrow b_1\dot w b_2 sb_2^{-1} \dot w^{-1}b_1^{-1}\in B\Leftrightarrow \dot w b_2 sb_2^{-1} \dot w^{-1}\in B\Leftrightarrow b_2 \in T.$$
Thus, if $s, gsg^{-1}\in B$ and $g\notin B$ then $g = b\dot w$ for
some $b \in B$. Then we have
\begin{equation}
[ b\dot w, h] = s \in T\label{equa3.8}
\end{equation}
for some $b \in B, s \in T$ and $\gamma \in G$. The assumption that
$g = b\dot w$ is an element of order $\ne 4$  implies that $b \ne
1$. Conjugating both sides of \eqref{equa3.8} with an appropriate
element $s^\prime\in T$, we can get
\begin{equation}
\left[\begin{pmatrix}\rho &1\cr
 -1&0\cr \end{pmatrix},
  \begin{pmatrix} \al&\be\cr
  \gam&\de\cr \end{pmatrix}\right] = \begin{pmatrix}\xi&0\cr 0&\xi^{-1}\cr
  \end{pmatrix}\label{equa3.9}
\end{equation}
for some $\rho , \al , \be , \gam ,\de , \xi\in \C$. A straightforward calculation shows
that
\begin{equation}
\left[\begin{pmatrix}\rho &1\cr
 -1&0\cr \end{pmatrix},
  \begin{pmatrix} \al&\be\cr
  \gam&\de\cr \end{pmatrix}\right] = \begin{pmatrix} \ast&\ast\cr
  -\al\gam - \de\be + \rho\be\gam&\ast\cr \end{pmatrix}\label{equa3.10}
\end{equation}
Then \eqref{equa3.9} and \eqref{equa3.10} imply that $ -\al\gam - \de\be + \rho\be\gam
=0$. Since $\al\de - \be\gam = 1$, we have
\begin{equation}
\begin{cases}-\al\gam - \de\be + \rho\be\gam=0\\ \al\de - \be\gam = 1\end{cases}\Rightarrow \begin{cases} \gam =- \frac{\be}{\al^2+\be^2- \rho\al\be}\\ \de
=\frac{\al-\rho\be}{\al^2+\be^2-\rho\al\be}\end{cases}.\label{equa3.11}
\end{equation}
(We omit the case $\al^2 +\be^2 -\rho\al\be =0$ because in this case we get $\be = 0$ and then $\gamma =0$. But then $h\in T$ is a semismple element and therefore
$$[b\dot w , h] = b\underbrace{\dot w h \dot w^{-1}}_{=h^{-1}} b^{-1} h^{-1} = bh^{-1}b^{-1} h^{-1} =\underbrace{[b, h^{-1}]}_{\notin T} h^{-2} \notin T.$$
But according to our assumption, $[g, h]\in T$.)

We now substitute \eqref{equa3.11} into \eqref{equa3.10} and obtain
\begin{equation}
\left[\begin{pmatrix}\rho &1\cr
 -1&0\cr \end{pmatrix},
  \begin{pmatrix} \al&\be\cr
  - \frac{\be}{\al^2+\be^2- \rho\al\be} & \frac{\al-\rho\be}{\al^2+\be^2- \rho\al\be}\cr \end{pmatrix}\right] = \begin{pmatrix} \ast&\rho (1 - (\al^2 +\be^2 - \rho\al\be))\cr
  0&\al^2+\be^2-\rho\al\be\cr \end{pmatrix}.\label{equa3.12}
\end{equation}
Comparing \eqref{equa3.9} and \eqref{equa3.12}, we get $\rho(1 - (\al^2
+\be^2 -\rho\al\be) )=0$. According to our assumption, $\rho\ne 0$. Therefore
$\al^2+\be^2 -\rho\al\be = 1$, and we get as the commutator in \eqref{equa3.12}
a matrix of the form
  $$ s = \begin{pmatrix}\xi&0\cr 0&\xi^{-1}\cr \end{pmatrix}= \begin{pmatrix}\ast&0\cr 0 &1\cr \end{pmatrix}.$$
Hence $s = 1$. This is a contradiction with our assumption. Thus,
$s, g \in B$. The lemma is proved.
\end{proof}

We now continue the proof of the theorem by showing that $\T_w$ has
only three irreducible components $\T_w^0, \T_w^1, \T_w^2$. Indeed,
assume that there is a component $\T_w^3$. Lemma \ref{lem81} implies
that a general pair $(g_1, g_2)\in \T^3_w$ satisfies the following
conditions: $g_1, [g_1, g_2]$ are semisimple elements where the
order of $g_1$ is not 1, 2, or 4 and $[g_1, g_2]\ne \pm 1$. Further,
since $\tr (w(g_1, g_2) ) = 2$, the element
$$w(g_1, g_2) =  [ [g_1, g_2], g_1[g_1,g_2]g_1^{-1}]$$
is unipotent, and the elements $[g_1, g_2], g_1[g_1,g_2]g_1^{-1}$
belong to the same Borel subgroup according to Lemma \ref{lem79}.
Then Lemma \ref{lem85} implies that $[g_1, g_2], g_1$ also belong to
the same Borel subgroup. We may assume $g_1 \in T$ and $[g_1, g_2] \in B$. Then applying
Lemma \ref{lem79} once again, we get $g_2 \in B$ or $g_2 \in B
\dot w$. Thus $(g_1, g_2) \in \T_w^0$ or  $(g_1, g_2) \in \T_w^1$ (note that the pair $(t , b\dot w)$, where $t \in T$ and $b \in B$, belongs to $\T_w^1$ because $\dot w(t, b\dot w) \dot w^{-1} = (t^{-1}, \dot w b)$). This contradicts the assumption
that $(g_1, g_2)$ is a general pair in $\T_w^3\ne\T_w^0,  \T_w^1$.

Since $\W_w^0 = \T_w^0, \W^2_w = \T_w^2$ and $\W_w^1$ is the only
irreducible component of $\W_w$ contained in $\T_w^1$, we conclude
that $\W_w$ also contains only three irreducible components. The
theorem is proved.
\end{proof}

\begin{cor}
Let $w$ be a word appearing in any of Examples \ref{ex1}, \ref{ex2},
\ref{ex3}, \ref{ex4}, \ref{ex5}. Then the induced map $\w\colon \PSL
(2,\mathbb C)^2\to\PSL (2,\mathbb C)$ is surjective.
\end{cor}

\begin{proof}
By Corollary \ref{cor2}, the image of $\w$ contains all semisimple
elements of $\PSL (2,\mathbb C)$. The computations of this section
provide a 4-dimensional component of $\W_w$ for each $w$, hence
guarantee that all unipotent elements lie in the image of $\w$. As
every element of $\PSL (2,\mathbb C)$ is either semisimple, or
unipotent, we are done.
\end{proof}

\begin{remark}
Note that to prove the corollary, we do not need the complete lists
of the irreducible components of the varieties $\W_w$ (whose
computation may be technically involved enough, as in Example
\ref{ex5}). It suffices to find a 4-dimensional component, which is
much easier.
\end{remark}

\begin{remark}
Towards computer-aided search of eventual examples of words $w$ such
that $G=\SL (2,\mathbb C)$ contains unipotents lying outside the
image of $\w$, it would be important to make the calculation of
irreducible components faster. Towards this end, it makes sense to
replace the representation variety $\W_w$ with the character variety
$\mathcal X_w=\W_w\sslash G$.
\end{remark}

\begin{remark}
An approach to proving that all unipotent elements lie in the
image of a word map $\w$ where $w \notin F_m^2$  for every simple algebraic group over fields of characteristic zero is based on the Magnus embedding theorem,
combined with the Jacobson--Morozov theorem, see \cite{BZ}.
\end{remark}

\begin{remark}
It is tempting to use representation varieties of associative and
Lie algebras (see, e.g., \cite[Remark 1.5]{Na}) to study the images
of polynomial maps on such algebras, with an eye towards solving some
problems raised in \cite{KBMR}, \cite{BGKP}, \cite{KBKP}.
\end{remark}

\bigskip

\noindent{\it Acknowledgements.} The research of the first author
was supported by the Ministry of Education and Science of the
Russian Federation and RFBR grant 14-01-00820. The research of the
second and third authors was supported by ISF grants 1207/12,
1623/16 and the Emmy Noether Research Institute for Mathematics.

We thank M. Borovoi for providing the reference \cite{GOV}. We are grateful
to the referee for careful reading and useful remarks.

\end{document}